\pdfoutput=1
\documentclass{amsart}
\usepackage[bbgreekl]{mathbbol}
\usepackage{dsfont}
\usepackage{tikz-cd}
\usepackage{mathtools}
\usepackage{helvet}
\usepackage{color}
\usepackage{mathrsfs,etex}
\usepackage{cancel}

\usepackage{amsmath,amsxtra,amsthm,amssymb, mathtools, tikz-cd, xr}
\usepackage[all]{xy}
\usepackage{hyperref}

\newcommand\robout{\bgroup\markoverwith {\textcolor{blue}{\rule[0.5ex]{2pt}{0.4pt}}}\ULon} 

\newtheorem{theorem}{Theorem}[subsection]
\newtheorem{lemma}[theorem]{Lemma}

\newtheorem{proposition}[theorem]{Proposition}
\newtheorem{corollary}[theorem]{Corollary}
\newtheorem{defn}[theorem]{Definition}

\theoremstyle{remark}
\newtheorem{remark}[theorem]{Remark}
\makeatletter
\newcommand{\mylabel}[2]{#2\def\@currentlabel{#2}\label{#1}}
\makeatother
\newcommand{\Gal}{\operatorname{Gal}}

\newcommand{\DD}{\mathbb{D}}

\newcommand{\NN}{\mathbb{N}}
\newcommand{\QQ}{\mathbb{Q}}
\newcommand{\Qp}{\mathbb{Q}_p}
\newcommand{\Zp}{\mathbb{Z}_p}
\newcommand{\ZZ}{\mathbb{Z}}

\newcommand{\FF}{\mathbb{F}}

\newcommand{\g}{\underline{\mathbf{g}}}

\newcommand{\ord}{\mathrm{ord}}

\newcommand{\fp}{\mathfrak{p}}

\newcommand{\cO}{\mathcal{O}}
\newcommand{\Iw}{\mathrm{Iw}}

\newcommand{\GL}{\mathrm{GL}}

\newcommand{\cyc}{\textup{cyc}}

\newcommand{\ubeta}{\bbbeta}
\newcommand{\scrA}{\mathscr{A}}

\newcommand{\DdagrigA}{\mathbb{D}^\dagger_{\textup{rig},\mathscr{A}}}
\newcommand{\DdagrigAA}{\mathbb{D}^\dagger_{\textup{rig},{A}}}
\newcommand{\Ddagrig}{\mathbb{D}^\dagger_{\textup{rig}}}

\newcommand{\LL}{\Lambda}

\newcommand{\RR}{\mathcal{R}}
\newcommand{\f}{\underline{\textup{\bf f}}}

\newcommand{\lra}{\longrightarrow}
\newcommand{\ra}{\rightarrow}
\newcommand{\res}{\textup{res}}

\newcommand{\fP}{\mathfrak{P}}

\newcommand{\Dcris}{\mathbb{D}_{\rm cris}}

\newcommand{\p}{\mathfrak{p}}
\newcommand{\q}{\mathfrak{q}}

\newcommand{\m}{\mathfrak{m}}

\newcommand{\Be}{}
\newcommand{\RS}{{\rm RS}}
\newcommand{\naive}{{\rm naive}}

\begin{document}

\title[$p$-adic Gross--Zagier at Critical Slope]
{$p$-adic Gross--Zagier formula at critical slope and a conjecture of Perrin-Riou}

\author[K. B\"uy\"ukboduk]{K\^az\i m B\"uy\"ukboduk}
\address[B\"uy\"ukboduk]{UCD School of Mathematics and Statistics\\ University College Dublin\\ Ireland}
\email{kazim.buyukboduk@ucd.ie}

\author{Robert Pollack}
\address[Pollack]{Department of Mathematics and Statistics, Boston University  \\
111 Cummington Mall, Boston, MA 02215 USA\\ }
\email{rpollack@math.bu.edu}

\author{Shu Sasaki}
\address[Sasaki]{Fakult\"at f\"ur Mathematik, Universit\"at Duisburg-Essen, Thea-Leymann Str. 9, 45127, Essen, Germany}
\email{s.sasaki.03@cantab.net}

\begin{abstract}
Let $p$ be an odd prime. Given an imaginary quadratic field $K=\QQ(\sqrt{-D_K})$ where $p$ splits with $D_K>3$, and a $p$-ordinary newform $f \in S_k(\Gamma_0(N))$ such that $N$ verifies the Heegner hypothesis relative to $K$, we prove a $p$-adic Gross--Zagier formula for the critical slope $p$-stabilization of $f$ (assuming that it is non-$\theta$-critical). In the particular case when $f=f_A$ is the newform of weight $2$ associated to an elliptic curve $A$ that has good ordinary reduction at $p$, this allows us to verify a conjecture of Perrin-Riou. The $p$-adic Gross--Zagier formula we prove has applications also towards the Birch and Swinnerton-Dyer formula for elliptic curves of analytic rank one. 
\end{abstract}
\maketitle
\tableofcontents

\section{Introduction}
\label{sec:Intro}
Fix forever an odd prime $p$ as well as embeddings $\iota_\infty:\overline{\QQ}\hookrightarrow \mathbb{C}$ and $\iota_p: \overline{\QQ}\hookrightarrow \overline{\QQ}_p$. Let $N$ be an integer coprime to $p$. We let $v_p$ denote the valuation on $\overline{\QQ}_p$, normalized so that $v_p(p)=1$. 

Let $f=\underset{n=1}{\overset{\infty}\sum} a_nq^n \in S_k(\Gamma_0(N))$ be a normalized newform of even  weight $k\geq 2$ and level $N\geq 3$. Let $K_f:=\iota_\infty^{-1}(\QQ(\cdots,a_n,\cdots))$ denote the Hecke field of $f$ and $\mathfrak{P}$ the prime of $K_f$ induced by the embedding $\iota_p$. Let $E$ denote an extension of $\QQ_p$ that contains $\iota_p(K_f)$. We shall assume that $v_p(\iota_p(a_p))=0$, namely that $f$ is $\fP$-ordinary. Let $\alpha,\beta\in \overline{\QQ}$ denote the roots of the Hecke polynomial $X^2-\iota_{\infty}^{-1}(a_p)X+p^{k-1}$ of $f$ at $p$. We assume that $E$ is large enough to contain both $\iota_p(\alpha)$ and $\iota_p(\beta)$. Since we assume that $f$ is $\fP$-ordinary, precisely one of $\iota_p(\alpha)$ and $\iota_p(\beta)$ (say, without loss of generality, $\iota_p(\alpha)$) is a $p$-adic unit. Then $v_p(\iota_p(\beta))=k-1$. To ease our notation, we will omit $\iota_p$ and $\iota_\infty$ from our notation unless there is a danger of confusion.

The $p$-stabilization $f^\alpha \in S_k(\Gamma_0(Np))$ of $f$ is called the ordinary stabilization and $f^\beta$ is called the critical-slope $p$-stabilization. We shall assume throughout that $f^\beta$ is not $\theta$-critical $($in the sense of \cite{bellaichepadicL}, Definition 2.12$)$.

Our main goal in the current article is to prove a $p$-adic Gross--Zagier formula for the critical-slope $p$-stabilization $f^\beta$. This is Theorem~\ref{thm:padicGZoverKIntro}. In the particular case when $f$ has weight $2$ and it is associated to an elliptic curve $A_{/{\QQ}}$, this result allows us to prove a conjecture of Perrin-Riou. This is recorded below as Theorem~\ref{thm:PRconjectureIntro}; it can be also translated into the statement of Theorem~\ref{thm:PRRubinFormula}, which is an explicit construction of a point of infinite order in $A(\QQ)$ in terms of the two $p$-adic $L$-functions associated to $f_A$ (under the assumption that $A$ has analytic rank one, of course). As a byproduct of Theorem~\ref{thm:padicGZoverKIntro}, we may also deduce that at least one of the two $p$-adic height pairings associated to $A$ is non-degenerate, still assuming that $A$ has analytic rank one. This fact yields the proof of the $p$-part of the Birch and Swinnerton-Dyer formula\footnote{Under the additional hypothesis that $A$ be semistable, this has been proved in \cite{BertiBertoliniVenerucci,JetchevSkinnerWan,ZhangBSD} using different techniques. We do not need to assume that $A$ is semistable.} for $A$; this is Theorem~\ref{thm:BSDformulainrankoneIntro} below.

Before we discuss these results in detail, we will introduce more notation. Let $S$ denote the set consisting of all rational primes dividing $Np$ together with the archimedean place. We let $W_f$ denote Deligne's (cohomological) $p$-adic representation associated to $f$ (so that the Hodge--Tate weights of $W_f$ are $(1-k,0)$, with the convention that the Hodge--Tate weight of the cyclotomic character is $+1$). Set $V_f=W_f(k/2)$; we call $V_f$ the central critical twist of $W_f$. Both $W_f$ and $V_f$ are unramified outside $S$ and they are crystalline at $p$. 

Let  $\Dcris (V_f)$ denote the crystalline Dieudonn\'e module and $\mathbb{D}^\dagger_{\textup{rig}}(V_f)$ Fontaine's (\'etale) $(\varphi,\Gamma_\cyc)$-module associated to $V_f\vert_{G_{\QQ_p}}$. We let $D_\alpha, D_\beta$ denote the eigenspaces of $\Dcris (V_f)$ for the action of the crystalline Frobenius $\varphi$; so that $\varphi\vert_{D_\alpha}=p^{-k/2}\alpha$ and $\varphi\vert_{D_\beta}=p^{-k/2}\beta$.

Let $K=\QQ(\sqrt{-D_K})$ be an imaginary quadratic field (where $D_K>3$ is a square-free integer congruent to $3$ modulo $4$) and let $H^1_{\rm f}(K,V_f)$ denote the Bloch-Kato Selmer group associated to $V_f$ over $K$. For each $\lambda\in \{\alpha,\beta\}$ the submodule $D_\lambda\subset \Dcris (V_f)$ defines a canonical splitting of the Hodge filtration on $\Dcris (V_f)$, namely that we have 
\[
\Dcris (V_f)=D_\lambda\oplus {\rm Fil}^0\Dcris (V_f).
\]
as $E$-vector spaces. Note that our assumption that $f^\beta$ is non-$\theta$-critical is necessary to ensure this splitting when $\lambda=\beta$ (see \cite{bellaichepadicL}, Proposition 2.11(iv)). Associated to this splitting and the cyclotomic $\ZZ_p$-extension of $K$, Nekov\'a\v r in \cite{Ne92} defined a $p$-adic height pairing
\[
h^{\textup{Nek}}_{\lambda,K}\,:\,H^1_{\textup{f}}(K,V_f)\times H^1_{\textup{f}}(K,V_f)\lra E\,.
\]
Suppose that the prime $p$ splits in $K$ and write $(p)=\p\p^c$. Assume also that $K$ verifies the Heegner hypothesis relative to $N$. Let $\epsilon_K$ denote the quadratic Dirichlet character associated to $K/\QQ$. The Heegner hypothesis ensures that $\ord_{s=\frac{k}{2}}L(f_{/K},s)$ is odd and there exists a Heegner cycle $z_f \in {H}_\textup{f}^1(K,V_f)$. 

\subsection{Results}
\label{subsec:results}
Let $L_{p,\beta}^{\rm Kob}(f_{/K},s)$ be the {$p$-adic $L$-function} given as in (\ref{DEF_TwoVarAmiceTransforms}). It is the critical slope counterpart of Nekov\'a\v{r}'s $p$-adic $L$-function associated to the $p$-ordinary stabilization $f^\alpha$. It follows from the interpolation property of $L_{p,\beta}^{\rm Kob}(f_{/K},s)$ that $L_{p,\beta}^{\rm Kob}(f_{/K},\frac{k}{2})=0$. As its predecessors, our $p$-adic Gross--Zagier formula expresses the first derivative of $L_{p,\beta}^{\rm Kob}(f_{/K},s)$ in terms of the $p$-adic height of the Heegner cycle $z_f$:
\begin{theorem}
\label{thm:padicGZoverKIntro}
Let $f\in S_{k}(\Gamma_0(N))$ be a newform with $N\geq 3$. Suppose $f$ is $p$-ordinary with respect to the embedding $\iota_p$ and let $f^\beta$ denote its critical-slope $p$-stabilization $($of slope $v_p(\beta)=k-1$$)$. Assume also that $f^\beta$ is not $\theta$-critical. Let $K=\QQ(\sqrt{-D_K})$ be an imaginary quadratic field where the prime $p$ splits and that satisfies the Heegner hypothesis relative to $N$. Then,
$$\displaystyle{\frac{d}{ds}L_{p,\beta}^{\rm Kob}(f_{/K},s)\big{|}_{s=\frac{k}{2}}={\left(1-\frac{p^{\frac{k}{2}-1}}{\beta}\right)^{4}}\cdot  \frac{h^{\textup{Nek}}_{\beta,K}(z_f,z_f)}{(4|D_K|)^{\frac{k}{2}-1}}}.$$
\end{theorem}

This theorem is proved (in \S\ref{subsec_proof_of_Cor_thm_padicGZoverKIntro}) by appealing to the existence of $p$-adic families of finite slope modular forms, which allows us, using the existence of a suitable two-variable $p$-adic $L$-function\footnote{The construction of this $p$-adic $L$-function follows from the work of Loeffler~\cite{LoefflerpadicRankinSelberg} and Loeffler-Zerbes~\cite{LZ1}.} and Theorem~\ref{thm:bigheegnermain}, to reduce to the case of non-critical slope. Such a non-critical slope result is precisely Theorem~\ref{thm:noncriticalGZhigherweight} ($p$-adic Gross-Zagier formula for non-ordinary eigenforms of arbitrary weight; which is work in progress by Kobayashi~\cite{kobayashi_higherweight_nonord_GZ}). More precisely, Kobayashi's method only establishes a $p$-adic Gross--Zagier formula in the non-ordinary case for one of the two $p$-stabilization of a given form (the one of smaller slope).  However, this result is sufficient for our method and moreover, our method not only yields the Gross--Zagier formula in the case of critical slope, but also allows us to handle the case of the other non-ordinary $p$-stabilization.  See Theorem~\ref{THM_GeneralizeKobayashiToBothPadicLFunctions} for an even more general statement.

 \subsubsection{Abelian varieties of $\GL_2$-type}\label{subsubsec_abvarGL2type}
 We assume until the end of this introduction that $f$ has weight $2$. Let $A_f/\QQ$ denote the abelian variety of $\GL_2$-type that the Eichler-Shimura congruences associate to $f$. This means that there exists an order $\cO_f\subset K_f$ and an embedding $\cO_f\hookrightarrow {\rm End}_\QQ(A_f)$. We shall assume that $\ord_{s=1}L(f_{/\QQ},1)=1$ and we choose $K$ (relying on \cite{BumpFriedbergHoffstein}) in a way to ensure that $\ord_{s=1}L(f_{/K},1)=1$ as well. In this scenario, the element $z_f \in H^1(\QQ,V_f)$ is obtained as the Kummer image of the $f$-isotypical component $P_f$ of a \emph{Heegner point}\footnote{More precisely, $z\in J_0(N)(K)$ is given as the trace of a Heegner point $y\in J_0(N)(H_K)$ which is defined over the Hilbert class field $H_K$ of $K$. Our restriction on the sign of the functional equation (for the Hecke $L$-function of $f$) shows that $P_f\in  A_f(\QQ)\otimes K_f$.} $z\in J_0(N)(K)$. Here, $J_0(N)$ is the Jacobian variety of the modular curve $X_0(N)$ over $\mathbb{Q}$ and we endow it with the canonical principal polarization induced by the intersection form on $H^1(X_0(N),\ZZ)$. This equips $A_f$ with a canonical polarization as well.
 
 We let $\langle \,,\,\rangle_{\infty}^{J_0(N)}$ denote the N\'eron-Tate height pairing on the abelian variety $J_0(N)$. Nekov\'a\v{r}'s constructions in \cite{Ne92} give rise to a pair of $E$-equivariant  $p$-adic height pairings
\[
h^{\textup{Nek}}_{\lambda,\QQ}\,:\,(A_f(\QQ)\otimes_{\cO_f}E)\times (A_f(\QQ)\otimes_{\cO_f}E)\lra E
\]
for each $\lambda=\alpha, \beta$. We set
$$c(f):=-\frac{L^\prime(f_{/\QQ},1)}{\langle P_f,P_f\rangle_{\infty}^{J_0(N)}\,\Omega_f^+}\in K_f^\times$$
where $\Omega_f^+$ is a choice of Shimura's period. We note that $K_f$-rationality of $c(f)$ is proved in \cite{GrossZagier1986}.

\begin{corollary}
\label{cor:padicGZoverQIntro} In addition to the hypotheses of Theorem~\ref{thm:padicGZoverKIntro}, suppose that $k=2$ and $\ord_{s=1}L(f_{/\QQ},1)=1$. Then for $P_f\in {A_f(\QQ)\otimes_{\cO_f} E}$ and $c(f)\in K_f^\times$ as in the previous paragraph we have
$$L_{p,\beta}^\prime(f_{/\QQ},1)= \left(1-{1}/{\beta}\right)^2c(f)\,  h^{\textup{Nek}}_{\beta,\QQ}(P_f,P_f).$$
\end{corollary}
Corollary~\ref{cor:padicGZoverQIntro} is proved in \S\ref{subsec_proof_of_Cor_cor_padicGZoverQIntro}.

\begin{remark}
\label{rem:pordversionsIntro}
The version of Theorem~\ref{thm:padicGZoverKIntro} above for the $p$-ordinary stabilization $f^\alpha$ is due to Perrin-Riou (when $k=2$) and Nekov\'a\v{r} (when $k$ is general). The version of Corollary~\ref{cor:padicGZoverQIntro} concerning the $p$-adic $L$-function $L_{p,\alpha}(f_{/\QQ},s)$ follows from Perrin-Riou's $p$-adic Gross--Zagier theorem.
\end{remark}

\subsubsection{Elliptic curves} In this subsection, we will specialize to the case when $K_f=\QQ$, so that $A=A_f$ is an elliptic curve defined over $\QQ$ of conductor $N$ and analytic rank one, with good ordinary reduction at $p$ and without CM. We note that it follows from \cite[Theorem 1.3]{EmertonTheta} that $f^\beta$ is not $\theta$-critical.

We still  assume that $\ord_{s=1}L(f_{/\QQ},1)=1$ and we choose $K$ as in \S\ref{subsubsec_abvarGL2type}. We assume that the mod $p$ representation 
$$\overline{\rho}_A: G_{\QQ}\lra {\rm Aut}_{\FF_p}(A[p])\stackrel{\sim}{\lra}\GL_2(\FF_p)$$
is absolutely irreducible. We fix a Weierstrass minimal model $\mathcal{A}_{/\ZZ}$ of $A$ and let $\omega_{\mathcal{A}}$ denote the N\'eron differential normalized so that that its associated real period $\Omega_A^+:=\int_{E(\mathbb R)}\omega_{\mathcal A}$ is positive.  Set $V=T_p(A)\otimes\QQ_p$ and we let $\omega_{\textup{cris}}\in \DD_{\textup{cris}}(V)$ denote the element that corresponds to $\omega_\mathcal{A}$ under the comparison isomorphism. We denote by $D_\alpha, D_\beta\subset \DD_{\rm cris}(V)$ the corresponding eigenspaces as before. Set $\omega_{\textup{cris}}=\omega_\alpha+\omega_\beta$ with $\omega_\alpha \in D_\alpha$ and $\omega_\beta\in D_\beta$. We let 
$$[-,-]: \DD_{\textup{cris}}(V)\times \DD_{\textup{cris}}(V)\lra E$$ 
denote  the canonical pairing (induced from the Weil pairing) and we set $\delta_A:=[\omega_\beta,\omega_\alpha]/c(f)$. We let $\omega_{\mathcal{A}}^*\in \DD_{\rm cris}(V)/{\rm Fil^0}\DD_{\rm cris}(V)$ denote the unique element such that $[\omega_{\mathcal{A}},\omega_{\mathcal{A}}^*]=1$. We remark that $\DD_{\rm cris}(V)/{\rm Fil^0}\DD_{\rm cris}(V)$ may be identified with the tangent space of $A(\QQ_p)$ and the Bloch-Kato exponential map 
$$\exp_V:\,\DD_{\rm cris}(V)/{\rm Fil^0}\DD_{\rm cris}(V)\lra H^1_{\rm f}(\QQ_p,V)=A(\QQ_p)\otimes\QQ_p$$ 
with the exponential map for the $p$-adic Lie group $A(\QQ_p)$. 
\begin{theorem}
\label{thm:PRRubinFormula}
Suppose $A=A_f$ is in the previous paragraph $($so $k=2$, $K_f=\QQ$ and $\overline{\rho}_f$ is absolutely irreducible$)$. In addition to all the hypotheses of Theorem~\ref{thm:padicGZoverKIntro}, assume that $\ord_{s=1}L(A_{/\QQ},1)=1$. Then
$$\exp_{V}\left(\omega^*_{\mathcal{A}}\cdot\sqrt{\delta_A\left((1-1/\alpha)^{-2}\cdot L_{p,\alpha}^\prime(f_{/\QQ},1)-(1-1/\beta)^{-2}\cdot L_{p,\beta}^\prime(f_{/\QQ},1)\right)}\right)$$
is a $\QQ$-rational point on the elliptic curve $A$ of infinite order.
\end{theorem}
The theorem above asserts the validity of a conjecture of Perrin-Riou. We also note that this theorem allows the explicit computation of rational points on elliptic curves.  Indeed one can compute the expression appearing in Theorem \ref{thm:PRRubinFormula} to very high $p$-adic accuracy by using the methods of \cite{PS} where algorithms are given to compute the derivatives of both  ordinary and critical slope $p$-adic $L$-functions.  Such computations should be compared to the analogous computations in \cite{KP} in the non-ordinary case.

Theorem \ref{thm:PRRubinFormula} may be deduced from the next result we present (in a manner identical to the argument in \cite[\S2.3]{kbbIwasawa2017}), which compares  the Bloch-Kato logarithms of two distinguished elements of the Bloch-Kato Selmer group $H^1_{\rm}(\QQ,V)$\,: the Beilinson-Kato element ${\rm BK}_1$ and the Heegner point $P_f$ given as above, for an appropriate choice of the imaginary quadratic field $K$. Notice that under our running hypotheses,
$$H^1_{\rm f}(\QQ,V)=A(\QQ)\otimes\QQ_p$$ 
and it is a one-dimensional $\QQ_p$-vector space. Note that a suitable integer multiple of $P_f\in A(\QQ)\otimes\QQ$ is a rational point on $A$ and as such, it is a genuinely algebraic object, whereas ${\rm BK}_1 \in A(\QQ)\otimes\QQ_p$ is a transcendental object that relates to both $p$-adic $L$-functions. The proof of Theorem~\ref{thm:PRRubinFormula} boils down to setting up an explicit comparison between ${\rm BK}_1$  and $P_f$. This is precisely the content of Theorem~\ref{thm:PRconjectureIntro}. It was conjectured by Perrin-Riou and was proved independently by Bertolini--Darmon--Venerucci in their preprint~\cite{BDV} (their approach is different from ours).

\begin{theorem}
\label{thm:PRconjectureIntro}
Suppose $A_{/\QQ}$ is an elliptic curve as in Theorem~\ref{thm:PRRubinFormula} and let $P\in A(\QQ)$ be any lift of a generator of $A(\QQ)/A(\QQ)_{\rm tor}$. We have
$$\log_A\left(\textup{res}_p({\rm BK}_1)\right)=-(1-1/\alpha)(1-1/\beta)\cdot d(f)\cdot\log_A\left(\textup{res}_p(P)\right)^2,$$
where $\log_A$ stands for the coordinate of the Bloch-Kato logarithm associated to $A$ with respect to the basis (of the tangent space) dual to that given by the N\'eron differential $\omega_{\mathcal{A}}$, and $d(f)=d^2c(f)$ with $d\in \QQ^\times$ the unique rational number with $P_f=P\otimes d$ in $A(\QQ)\otimes \QQ$.

\end{theorem}
The proof of Theorem~\ref{thm:PRconjectureIntro} is presented in \S\ref{sec:PRsconj}. One key result that we rely on establishing Theorem~\ref{thm:PRconjectureIntro} is the following consequence of our $p$-adic Gross--Zagier formula. We record it in \S\ref{subsubsec_padicheightsandBSDabvar} as we believe that it is of independent interest; while we re-iterate that a proof of Theorem~\ref{thm:PRRubinFormula} is not written down explicitly in this article as it follows verbatim as in~\cite{kbbIwasawa2017}.

\subsubsection{$p$-adic heights on elliptic curves and the conjecture of Birch and Swinnerton-Dyer} \label{subsubsec_padicheightsandBSDabvar}
Throughout this section, we still assume that $f\in S_2(\Gamma_0(N))$ and that $K_f=\QQ$. Let $A=A_f$ denote the rational elliptic curve associated to $f$. 
\begin{theorem}
\label{thm:padicheightnontrivialIntro}
Suppose $f=\sum a_nq^n\in S_2(\Gamma_0(N))$ is a normalized newform with $N\geq 3$, $K_f=\QQ$ and such that 
\begin{itemize}
\item $p\nmid a_p$, 
\item neither of the $p$-stabilizations of $f$ is $\theta$-critical (equivalently, $A$ is non-CM), 
\item the residual representation $\overline{\rho}_f$ is absolutely irreducible,
\item $\ord_{s=1}L(f_{/\QQ},1)=1$\,.
\end{itemize}
Then either $h^{\textup{Nek}}_{\alpha,\QQ}$ or $ h^{\textup{Nek}}_{\beta,\QQ}$ is non-degenerate.
\end{theorem}
Theorem~\ref{thm:padicheightnontrivialIntro} is proved in \S\ref{subsubsec_proof_of_thm_padicheightnontrivialIntro}.
\begin{remark}
Suppose $p$ is a prime of good supersingular reduction for the elliptic curve $A=A_f$, a stronger form of Theorem~\ref{thm:padicheightnontrivialIntro} was proved by Kobayashi in \cite{KobayashiGZ}. Fortunately, this weaker version is good enough for applications towards the Birch and Swinnerton-Dyer conjecture we discuss below.
\end{remark} 

The final result we shall record in this introduction (Theorem~\ref{thm:BSDformulainrankoneIntro} below) is a consequence of Theorem~\ref{thm:padicGZoverKIntro} and Theorem~\ref{thm:padicheightnontrivialIntro} towards the Birch and Swinnerton-Dyer conjecture for an elliptic curve $A/\QQ$. Under the additional hypothesis that $A$ be semistable, this has been proved in \cite{BertiBertoliniVenerucci, JetchevSkinnerWan, ZhangBSD}\footnote{Besides the assumption that $A$ be semistable, \cite[Theorem 7.3]{ZhangBSD} has additional assumption that $p$ is coprime to Tamagawa factors and \cite[Theorem A]{BertiBertoliniVenerucci} requires $p$ be non-anomalous for $A$. In \S5.6 of \cite{BertiBertoliniVenerucci}, the authors explain a strategy to weaken the semistability hypothesis.} using different techniques. Our results here allow us to adapt the proof of \cite[Cor. 1.3]{KobayashiGZ} to the current setting to obtain a much simpler proof of (the $p$-part of) the Birch and Swinnerton-Dyer formula and eliminate the semistability hypothesis in \cite{JetchevSkinnerWan}. The proof of Theorem~\ref{thm:BSDformulainrankoneIntro} is presented in \S\ref{sec:BSD}. 

We note that our argument in \S\ref{sec:BSD} can be extended (with appropriate adjustments) to treat the case of abelian varieties of ${\rm GL_2}$-type, once the easy generalizations of \cite[\S3]{PerrinRiou93Rubinsformula} and \cite[\S2]{kbbIwasawa2017} to cover the case when $K_f\neq \QQ$ become available.

 Let $f\in S_2(\Gamma_0(N))$ be the normalized newform associated to $A$ as above. We retain our hypothesis that $\ord_{s=1}L(f_{/\QQ},1)=1$. In this situation, it follows from the work of Gross--Zagier and Kolyvagin (on choosing the auxiliary imaginary quadratic field $K$ as in \S\ref{subsubsec_abvarGL2type}) that the Tate-Shafarevich group ${\rm III}(A/\QQ)$ is finite and the Heegner point $P_f$ generates the $\QQ$-vector space $A(\QQ)\otimes \QQ$.

\begin{theorem}
\label{thm:BSDformulainrankoneIntro} Suppose $A_{/\QQ}$ is a {non-CM} elliptic curve with analytic rank one and that 
\begin{itemize}
\item[\mylabel{item_MC1}{{\bf (MC1)}}] $A$ has good ordinary reduction at $p$,
\item[\mylabel{item_MC2}{{\bf (MC2)}}] $\overline{\rho}_A$ is absolutely irreducible, 
\item[\mylabel{item_MC3}{{\bf (MC3)}}] one of the following two conditions hold:
\begin{itemize}
\item[\mylabel{item_MC31}{{\bf (MC3.1)}}]  There exists a prime $q||N$ such that $p\nmid {\rm ord}_q(\Delta_q)$ for a minimal discriminant $\Delta_q$ of $A$ at $q$.
\item[\mylabel{item_MC32}{{\bf (MC3.2)}}] We have ${\rho}_A(G_\QQ)\supset {\rm SL}_2(\ZZ_p)$ and there exists a real quadratic field $F$ verifying the conditions of \cite[Theorem 4]{Wan2015ForumSigma}.
\end{itemize}
\end{itemize}
Then the $p$-part of the Birch and Swinnerton-Dyer formula for $A$ holds true. 
\end{theorem}

\begin{remark}
\label{intro_remark_IMC}
 The hypotheses \ref{item_MC1}--\ref{item_MC3} in Theorem~\ref{thm:BSDformulainrankoneIntro} are to insure that the Iwasawa main conjecture for the elliptic curve $A/\QQ$ at the good ordinary prime $p$ holds true; c.f. \cite{skinnerurbanmainconj}, \cite[Theorem 2.5.2]{skinnerPasificJournal2016} and \cite[Theorem 4]{Wan2015ForumSigma}.
\end{remark}

We close this introduction with a brief overview of our strategy to prove Theorem~\ref{thm:padicGZoverKIntro}. We remark that the original approach of Perrin-Riou and Kobayashi (which is an adaptation of the original argument of Gross and Zagier) cannot be applied in our case of interest as there is no Rankin-Selberg construction of the critical-slope $p$-adic $L$-functions $L_{p}(f^\beta_{/\QQ},s)$ and $L_{p}(f^\beta_{/\QQ}\otimes\epsilon_K,s)$. The main idea is to prove a version of the asserted identity in $p$-adic families. That is to say, we shall choose a Coleman family $\f$ through the $p$-stabilized eigenform $f^\beta$ (over an affinoid domain $\mathscr{A}$) and we shall consider the following objects that come associated to $\f$:
\begin{itemize}
\item A two-variable $p$-adic $L$-function $L_p(\f_{/K},s)$. The construction is essentially due to Loeffler (and it compares to that due to Bella\"iche); we recall its defining properties in \S\ref{sec:padicLfunctions} below. One subtle point is that this $p$-adic $L$-function does not\footnote{In fact, it could not: See {Remark~\ref{REM_KobayashiNekDoesNotInterpolate}} below where we explain that $L_{p,\beta}^{\rm Kob}(f_{/K},s)$ does not vary continuously as $f^\beta$ varies in families.} interpolate $L_{p,\beta}^{\rm Kob}(f_{/K},s)$, but rather an explicit multiple of it. This extra (non-interpolatable) $p$-adic multiplier is essentially the $p$-adic interpolation factor for the adjoint $p$-adic $L$-function attached to $f^\beta$. Crucially, the same factor also appears in the height side. 
\item An $\mathscr{A}$-adic height pairing $\mathfrak{h}_{\f,K}$ that interpolates Nekov\'a\v{r}'s $p$-adic height pairings for the members of the Coleman family, in the sense that the diagram \eqref{eqn:comaprisonofheights} below (located just before the start of \S\ref{subsec:BigHeegstatementsconsequences}) commutes. It is important to compare the ``correction factor''\footnote{This factor appears as the ratio of the two Poincar\'e duality pairings on the the $f$-direct summand summands of two modular curves of respective levels $\Gamma_0(N)$ and $\Gamma_0(N)\cap \Gamma_1(p)$. See Proposition~\ref{prop:comparisionoftwopoincarepairings} where we make this discussion precise. We are grateful to D. Loeffler for explaining this to us.} that appears on the right most vertical arrow in the lower right square to the non-interpolatable $p$-adic multiplier mentioned in the previous paragraph. The construction of the $\mathscr{A}$-adic height pairing is due to Benois and it is recalled in \S\ref{subsec:Aadicheights} below. 
\item A ``universal'' Heegner point $\mathscr{Z}_{\f}$ that interpolates the Heegner cycles associated to the central critical twists of the members of the family $\f$. The construction of this class is one of the main ingredients here and it is carried out in \cite{JLZHC, BLGHCpatching}.
\end{itemize}
Relying on the density of non-critical-slope crystalline points in the family $\f$ and a $p$-adic Gross--Zagier formula for these members (recorded in Theorem~\ref{thm:noncriticalGZhigherweight}, which is Kobayashi's work in progress), one may easily deduce an $\mathscr{A}$-adic Gross--Zagier formula\footnote{See also~\cite{DisegniUniversal} where a similar formula for slope-zero families was proved independently.} for $L_p(\f_{/K},s)$, expressing its derivative with respect to the cyclotomic variable as the $\mathscr{A}$-adic height of the universal Heegner cycle (see Theorem~\ref{thm:twovarGZ} below). The  proof of Theorem~\ref{thm:padicGZoverKIntro} follows, on specializing this statement to weight $k$.

Let $g=\sum_{n=1}a_n(g)q^n\in S_{2r}(\Gamma_0(N))$ be a normalized eigenform. We let $a,b\in \overline{\QQ}$ denote the roots of its Hecke polynomial $X^2-a_p(g)X+p^{2r-1}$ at $p$. Suppose that $v_p(\iota_p(a_p(g)))>0$ and assume that
$$0<v_p(\iota_p(b))\leq v_p(\iota_p(a))\,.$$
 Let $g^b \in S_{2r}(\Gamma_0(Np))$ denote the $p$-stabilization corresponding to the Hecke root $b$. Kobayashi's forthcoming result (Theorem~\ref{thm:noncriticalGZhigherweight} below) proves a $p$-adic Gross--Zagier formula for the $p$-stabilization $g^b$ alone. This is sufficient for our purposes; moreover, the method we present here (without any modification whatsoever) allows one to deduce the following $p$-adic Gross--Zagier formula (Theorem~\ref{THM_GeneralizeKobayashiToBothPadicLFunctions} below) at every non-$\theta$-critical point $x$ on the eigencurve of tame level $N$, that admits a neighborhood with a dense set of crystalline classical points (e.g. any crystalline non-$\theta$-critical classical point $x$ verifies this property).
 \begin{theorem}
\label{THM_GeneralizeKobayashiToBothPadicLFunctions}
Suppose $x$ is any non-$\theta$-critical point of weight $w$ on the eigencurve of tame level $N
\geq 3$, that admits a neighborhood  with a dense set of crystalline classical points. Set $L_p(x,s):=L_p^{\RS}(\mathbf{F},\kappa,s)\vert_x$, where $\mathbf{F}$ is any Coleman family over a sufficiently small neighborhood of $x$ and finally $L_p^{\RS}(\mathbf{F},\kappa,s)$ is as in Definition~\ref{THMloefflersinterpolation}. Then,
$$\frac{d}{ds}L_p(x,s)\vert_{s=\frac{w}{2}}=H_{x,K}(\mathscr{Z}_{x},\mathscr{Z}_{x})\,.$$
Here, $H_{x,K}$ is the specialization of the height pairing $H_{\mathbf{F},K}$ (given as in Definition~\ref{DEF_AmiceTransformOfUniversalHeight}) to $x$ and likewise, $\mathscr{Z}_x$ is the specialization of the universal Heegner cycle $\mathscr{Z}_{\mathbf{F}}$ to $x$. In particular, if $f\in S_k(\Gamma_0(N))$ is a classical newform and $\lambda$ is a $\varphi$-eigenvalue on $\Dcris (V_f)$ such that $f^\lambda$ is non-$\theta$-critical, then 
$$\displaystyle{\frac{d}{ds}L_{p,\lambda}^{\rm Kob}(f_{/K},s)\big{|}_{s=\frac{k}{2}}={\left(1-\frac{p^{\frac{k}{2}-1}}{\lambda}\right)^{4}}\cdot  \frac{h^{\textup{Nek}}_{\lambda,K}(z_f,z_f)}{(4|D_K|)^{\frac{k}{2}-1}}}.$$
\end{theorem}

\subsection*{Acknowledgements} The authors thank Daniel Disegni, Shinichi Kobayashi, David Loeffler and Barry Mazur for very helpful conversations. K.B.\ was hosted at Harvard University while preparing this article and its companion. He thanks the Mathematics Department at Harvard University for their amazing hospitality. R.P.\ thanks MPIM-Bonn for their strong support and hospitality throughout his extended visit there. S.S. thanks Vytas Pa\v{s}k$\bar{\mbox{u}}$nas for his unflagging moral support and helpful conversations.  
K.B.\ has received funding from the European Union's Horizon 2020 research and innovation programme under the Marie Sk\l odowska-Curie Grant Agreement No.\ 745691 (CriticalGZ).  R.P.\ acknowledges support from NSF grant DMS-1702178 as well as a fellowship from the Simons Foundation. S.S.\ acknowledges financial support from DFG/SFB. The authors thank the anonymous referees for their comments and suggestions, which guided us towards a significant improvement on an earlier version of this article.

\section{Notation and Set up}
\label{subsec:notationsetup}
For any field $L$, we let $\overline{L}$ denote a fixed separable closure and let $G_L:=\Gal(\overline{L}/L)$ denote its Galois group. 

For each prime $\lambda$ of a number field $F$, we fix a decomposition group at $\lambda$ and identify it with $G_{\lambda}:=G_{F_{\lambda}}$. We denote by $I_{\lambda}\subset G_\lambda$ the inertia subgroup. In the main body of our article, we will only work with the case when $F=\QQ$ or $F=K$ (the imaginary quadratic field we have fixed above). For any finite set of places $S$ of $F$, we denote by $F_S$ the maximal extension of $F$ unramified outside $S$ and set $G_{F,S}:=\Gal(F_S/F)$.

We set $\mathbb{C}_p:=\widehat{\overline{\QQ}}_p$, the $p$-adic completion of $\overline{\QQ}_p$. We fix embeddings $\iota_\infty: \overline{\QQ}\hookrightarrow \mathbb{C}$ and $\iota_p: \overline{\QQ}\hookrightarrow \mathbb{C}_p$. When the prime $p$ is assumed to split in the imaginary quadratic field $K$, we let $\fp$  denote the prime of $K$ corresponding to the embedding $\iota_p$.

We denote by $v_p\,:\,\mathbb{C}_p \ra \mathbb R \cup\{+\infty\}$ the $p$-adic valuation on $\mathbb{C}_p$ which is normalized by the requirement that $v_p(p)=1$. Set $\vert x\vert_p=  p^{-v_p(x)}.$ 
 
 We fix a system $\varepsilon=(\zeta_{p^n})_{n\geqslant 1}$ of primitive $p^n$th roots of the unity in $\overline{\QQ}$ such that $\zeta_{p^{n+1}}^p=\zeta_{p^n}$ for all $n$. We set  $\Gamma_\cyc=\mathrm{Gal}(\QQ(\zeta_{p^{\infty}})/\mathbb Q)$ and denote by 
 $$\chi_\cyc:\,\Gamma_\cyc \stackrel{\sim}{\lra} \ZZ_p^\times$$ 
 the cyclotomic character. 
The group $\Gamma_\cyc$ factors canonically as
$\Gamma_\cyc=\Delta\times \Gamma$ where $\Delta =\mathrm{Gal}(\QQ(\zeta_p)/\QQ)$
and $\Gamma=\mathrm{Gal}(\QQ(\zeta_{p^{\infty}})/\mathbb Q(\zeta_p))$. We let $\omega$ denote the Teichm\"uller character (that factors through $\Delta$) and set $\left\langle \chi_\cyc\right\rangle:=\omega^{-1}\chi_\cyc$. We let $\LL:=\ZZ_p[[\Gamma]]$. We write $\Lambda^\iota$ to denote the free $\LL$-module of rank one, on which $G_\QQ$ acts via 
\begin{align*} 
G_\QQ \twoheadrightarrow \Gamma&\stackrel{\iota}{ \lra} \Gamma\hookrightarrow \LL^\times\\
\iota:\gamma&\longmapsto \gamma^{-1}
\end{align*}

By slight abuse of notation, we denote all the objects ($\Gamma_\cyc,\chi_\cyc, \Delta,\Gamma, \omega,\LL$ and $\iota$) introduced in the previous paragraph but defined over the base field $\QQ_p$ (in place of $\QQ$) with the same set of symbols.

For any a topological group $G$ and a module $M$ that is equipped with a continuous $G$-action, we shall write $C^{\bullet}(G,M)$ for the complex of continuous cochains of $G$ with coefficients in $M$.

Let $S$ be a finite set of places of $\QQ$ that contains $p$ and the archimedean place. If $V$ is a $p$-adic representation of $G_{\QQ,S}$ with coefficients in an affinoid algebra ${A}$, we shall denote by $\DdagrigAA(V)$ the $(\varphi,\Gamma_\cyc)$-module associated to the restriction of $V$ to the decomposition group at $p$.

Let $\Sigma$ denote the set of rational primes that divides $Np$, together with the archimedean place. We denote the set of places of $K$ above those in $\Sigma$ also by $\Sigma$.

For our fixed imaginary quadratic field $K$, we let $\cO$ denote the maximal order of $K$. For any positive integer $c$, let $\cO_c:=\ZZ+c\cO$ denote the order of conductor $c$ in $K$ and let $H_c$ denote the ring class field of $K$ of $\cO_c$. Write $H_c^{(Np)}$ for the maximal extension of $H_c$ outside $Np$ and $\mathfrak{G}_c:=\Gal(H_c^{(Np)}/H_c)$. Fix a positive integer $c$ coprime to $N$. We also set $L_{cp^s}:=H_{cp^s}(\mu_{p^s})$.

For any eigenform $g$, we shall write $g^K$ in place of $g\otimes\epsilon_K$ for its twist by the quadratic character $\epsilon_K$ associated to $K/\QQ$.

For each non-negative real number $h$, we let $\mathcal{D}_{h}$ denote the $\QQ_p$-vector space of $h$-tempered distributions on $\ZZ_p$ and set $\mathcal{D}_{\infty}:=\cup_{h} \mathcal{D}_{h}$. We also let $\mathcal{D}$ denote the $\LL$-algebra of $\QQ_p$-valued locally analytic distributions on $\ZZ_p$. The natural map $\mathcal{D}_h\ra \mathcal{D}$ is an injection (for every $h$) since locally analytic functions are dense in the space of continuous functions.

We let $\mathscr{R}_+$ denote the $\QQ_p$-algebra of analytic functions on the open unit ball. In explicit terms, 
$$\mathscr{R}_+:=\left\{\sum_{n=0}^\infty c_n X^n: \lim_{n\ra \infty}|c_n|_p s^n=0 \hbox{ for every } s\in[0,1)\right\}.$$
According to \cite[Proposition 1.2.7]{PRLocalIwasawa1994}, the algebra $\mathcal{D}$ is naturally isomorphic via the Amice transform to $\mathscr{R}_+$. 

On fixing a topological generator $\gamma$ of $\Gamma$ (which in turn fixes isomorphisms $\Gamma\cong \ZZ_p$ and $\LL\cong \ZZ_p[[X]]$), we may define the $\QQ_p$-algebras
$$\mathcal{D}_{h}(\Gamma)\subset \mathcal{D}_{\infty}(\Gamma)\subset \mathcal{D}(\Gamma)$$
of distributions on $\Gamma$. We also set 
$$\mathscr{H}:=\{f(\gamma-1): f\in \mathscr{R}_+\}\subset \QQ_p[[\Gamma]]$$
(so that $\mathscr{H}\cong \mathscr{R}_+$ via $\gamma\mapsto 1+X$). For $H=\sum_{n=0}^{\infty}c_n(H)(\gamma-1)^n\in \mathscr{H}$, we shall set 
${\displaystyle H^\prime:=\sum_{n=0}^{\infty}\frac{(n+1)c_{n+1}(H)}{\log_p\chi_\cyc(\gamma)}(\gamma-1)^n\in \mathscr{H}}$. Notice then that $\mathds{1}(H^\prime)={c_1(H)}\big{/}{\log_p\chi_\cyc(\gamma)}$ does not depend on the choice of $\gamma$.

We shall equip $\mathcal{D}_h(\Gamma)$ ($0\leq h\leq \infty$) and $\mathscr{H}$ with a $\LL$-module and Galois module structure via the compositum of the maps
$$\LL\hookrightarrow \LL[1/p]=\mathcal{D}_0(\Gamma)\hookrightarrow \mathcal{D}(\Gamma)\stackrel{\rm{Amice}}{\hookrightarrow} \mathscr{H}\,.$$

We set $\mathscr{H}_{\mathscr{A}}:=\mathscr{H}\widehat{\otimes}\mathscr{A}$ given a $\QQ_p$-affinoid $\mathscr{A}$. We let $\mathscr{H}^\iota:=\mathscr{H}\otimes_{\LL}\LL^\iota$ and similarly define $\mathscr{H}_{\mathscr{A}}^\iota$.

\subsection{Modular curves, Hecke correspondences and the weight space}
\label{subsubsec:modularcurvesheckeweight}

For each non-negative integer $s\in \ZZ_{\geq2}$, we let $Y_s$ denote the affine modular curve over $\mathbb{Q}$ of level $\Gamma_0(N)\cap \Gamma_1(p^s)$. It parametrizes triples $(E,C,\varpi)$ where $E$ is an elliptic curve, $C$ is a cyclic group of $E$ of order $N$ and $\varpi$ is a point of order $p^s$. We let $X_s{}_{/\QQ}$ denote its compactification and $J_s:={\rm Jac}(X_s)$. 

For each $s$, we let $\mathfrak{H}_{s} \subset {\rm End}(J_s)$ denote $\ZZ_p$-the algebra generated by all Hecke operators $\{T_{\ell}\}_{\ell\nmid Np}$ together with $\{U_\ell\}_{\ell\mid Np}$ and the diamond operators $\{\langle m \rangle: m\in (\ZZ/p^s\ZZ)^\times\}$. 

We set $\LL_{\rm wt}:=\ZZ_p[[\ZZ_p^\times]]$.  
For $z\in \ZZ_p^\times$, we let $[z]\in \LL_{\rm wt}$ denote the group-like element. The Hecke algebra $\mathfrak{H}_{s}$ comes equipped with a $\LL_{\rm wt}$-module structure via $[z]\mapsto \langle z \rangle$. We let $\mathfrak{m}_s$ denote the maximal ideal of $\mathfrak{H}_s$  that is determined by the residual representation $\overline{\rho}_f$ associated to our fixed eigenform $f$. When there is no risk of confusion, we shall abbreviate $\mathfrak{m}:=\mathfrak{m}_s$.

 Following \cite{how2}, we define the critical weight character $\Theta:\Gamma_\cyc\ra \LL_{\rm wt}$ (centered at weight $k$) by setting
$$\Theta(\sigma):=\omega^{\frac{k}{2}-1}(\sigma)[\langle\chi_\cyc\rangle^{\frac{1}{2}}(\sigma)]$$
for $\sigma\in \Gamma_\cyc$, where $\omega:\Gamma_\cyc \ra \ZZ_p^\times$ is the Teichm\"uller character. We let $\LL_{\rm wt}^{\dagger}$ denote $\LL_{\rm wt}$ as a module over itself, but allowing $G_\QQ$ act via the character $\Theta^{-1}$. Let $\xi\in \LL_{\rm wt}^\dagger$ denote the element that corresponds to $1 \in \LL_{\rm wt}$.

For any $\mathfrak{H}_s$-module $M$ on which $G_\QQ$ acts, we shall write 
$$M^{\dagger}:=M\otimes_{\LL_{\rm wt}}\LL^{\dagger}_{\rm wt}$$ 
which we equip with the diagonal $G_\QQ$-action. Here the tensor product is over $\LL_{\rm wt}$ and its action on $M$ is given via the morphism $\LL_{\rm wt}\ra \mathfrak{H}_s$ (the diamond action).

\subsection{The Coleman family ($\hbox{\underline{\bf{f}}},\ubeta$)}
\label{subsubsec:colemanfamilyintro}
We fix an isomorphism $e_{k-2}\LL_{\rm wt} \cong\ZZ_p[[w]]$ and let $\mathscr{W}$  denote the weight space over $\QQ_p$ parametrizing continuous characters of $\Zp^\times$ (so that for an affinoid $\Qp$-algebra $A$, we have $\mathscr{W}(A)={\rm Hom}_{\rm cts}(\ZZ_p^\times,A^\times)$). Let $U=\overline B(k,p^{-a})\subset \mathscr W$ denote the closed disk about $k$ of radius $p^{-a}$ for some positive integer $a$, then $U$ has the structure of an affinoid space over $\QQ_p$.  We let $\mathcal O(U)$ denote the $E$-algebra of analytic functions on the affinoid $U$; the $E$-algebra $\cO(U)$  is isomorphic to the Tate algebra $\mathscr{A}=E\left\langle\left \langle\displaystyle{w}/{p^r}\right \rangle\right \rangle.$ 

For each $\kappa \in k+p^{r-1}\ZZ_p$, we shall denote by $\psi_{\kappa}$ the morphism 
\begin{align*}
\psi_\kappa:\, \scrA&\lra E\\
w &\longmapsto  (1+p)^{\kappa-k}-1\,.
\end{align*}

Consider the sequence $I=\{ \kappa \in \ZZ_{\geq 2} \mid \kappa\equiv k\pmod{(p-1)p^{r-1}}\}$ of integers and let
$$\f=\underset{n=1}{\overset{\infty}\sum} \mathbf{a}_nq^n \in \mathscr{A}[[q]]$$
denote a $p$-adic family of cuspidal eigenforms passing through $f^\beta$, in the sense of Coleman~\cite{coleman}. 
This means that for every  point  $\kappa\in I$, the formal expression
$${\f}(\kappa):=\underset{n=1}{\overset{\infty}\sum} \psi_{\kappa}(\mathbf{a}_n)q^n$$ 
is the $q$-expansion of a cuspidal eigenform of level $\Gamma_0(Np)$ and weight $\kappa$, with the additional property that ${\f}(k)=f.$ Let us denote by $\ubeta=\mathbf{a}_p$ for the $U_p$-eigenvalue for its action on $\f$, so that we have $\ubeta(k)=\beta$.  By shrinking the closed ball $U$ if necessary  and using \cite[Corollary B5.7.1]{coleman}, we may (and we will henceforth) assume that ${\f}$ is a family of constant slope $k-1$ (in particular, ${{\f}}$ specialises to a classical form of weight $w$ and slope $k-1$ at any integer weight $w>k$ lying in $U$).

Let $W_{\f}$ denote the big Galois representation associated to the family ${\f}$ with coefficients in $\scrA=\mathcal O(U)$. We define its twist $V_{\f}:=W_{\f}\otimes_{\LL_{\rm wt}}\LL_{\rm wt}^{\dagger}$. We recall that $W_{\f}$ comes equipped with a $\LL_{\rm wt}$-module structure via the diamond action. Note then that $V_{\f}$ is self-dual in the sense that we have a $G_{\QQ,\Sigma}$-equivariant symplectic pairing (that we denote by $\langle\,,\,\rangle_{Np^\infty}$)
\begin{equation}
\label{pairingforVf}
\langle\,,\,\rangle_{Np^\infty}\,:\,V_{{\f}}\times V_{{\f}} \lra \scrA(1)\,.
\end{equation}  

\section{Selmer complexes and $p$-adic heights in families}
\label{sec:padicheights}

\subsection{Cohomology of $(\varphi,\Gamma_\cyc)$-modules}
In this subsection, we shall review the cohomology of $(\varphi,\Gamma_\cyc)$-modules.  Fix a topological generator $\gamma$ of  $\Gamma$. Recall that $\mathscr{A}$ stands for the affinoid algebra over $E$ and $\RR_{\mathscr A}$ for the relative Robba ring over $\mathscr{A}$. For any $(\varphi,\Gamma_\cyc)$-module $\mathbb D$ over $\RR_{\mathscr{A}}$, consider the Fontaine--Herr complex    
\[
C^{\bullet}_{\varphi,\gamma}(\mathbb D)\,\,:\,\, 
\mathbb{D}^{\Delta}  \xrightarrow{d_0} 
\mathbb{D}^{\Delta} \oplus \mathbb{D}^{\Delta} 
\xrightarrow{d_1} \mathbb{D}^{\Delta},
\]
where  $d_0(x)=((\varphi-1)(x), (\gamma-1)x)$ and $d_1(y,z)=(\gamma-1)(y)-(\varphi-1)(z)$
(for further details and properties, see \cite{H1, Liu, KPX}). We define 
\[
H^i(\mathbb D):=H^i(C^{\bullet}_{\varphi,\gamma}(\mathbb D)).
\] 
It follows from \cite[Theorem 0.2]{Liu} and \cite[Theorem 4.4.2]{KPX} that $H^i(\mathbb D)$ is a finitely  generated  $\mathscr{A}$-module for $i=0,1,2$. 

In the particular case when $\mathbb D=\DdagrigA (V_{\f})$, it follows by \cite[Theorem 0.1]{Liu} and \cite[Theorem 2.8]{Pot13}) that there exist canonical (up to the choice of $\gamma$) and functorial isomorphisms
\begin{equation}
\label{formula:computationgaloiscohviaphigamma}
H^i(\DdagrigA (V_{\f}))\simeq H^i(\QQ_p,V_{\f}).
\end{equation}
for all $i$. The following proposition is due to Benois~ \cite[Proposition 2.4.2]{Ben14b} and refines the isomorphism \eqref{formula:computationgaloiscohviaphigamma}. Set
$$K_p^{\bullet}(V_{\f}):=\mathrm{Tot} \left (C^{\bullet}\left(G_p,V_{\f}\otimes_{\mathscr{A}}\widetilde{\mathbb{B}}_{\rm{rig},\mathscr{A}}^{\dagger}\right) \xrightarrow{\varphi-1} C^{\bullet}\left(G_p,V_{\f}\otimes_{\mathscr{A}}\widetilde{\mathbb{B}}_{\rm{rig},\mathscr{A}}^{\dagger}\right)
\right ),$$
where $\widetilde{\mathbb{B}}_{\rm{rig},\mathscr{A}}^{\dagger}$ is the ring of $p$-adic periods introduced by Berger in \cite{Ber02}. 

\begin{proposition}[Benois]
\label{proposition:quasi-isomorphisms} We have a diagram 
\begin{equation*}
\xymatrix{
C^{\bullet}(G_p,V_{\f}) \ar[r]^{\xi}_{\simeq} &K_p^{\bullet}(V_{\f}),\\
& C^{\bullet}_{\varphi,\gamma}(\Ddagrig (V_{\f})) \ar[u]^{\eta}_{\simeq}}
\end{equation*}
where the maps $\eta$ and $\xi$ are both quasi-isomorphisms.
\end{proposition}

\subsection{Selmer complexes}
\label{subsec:selmercomplexes}
\subsubsection{Local conditions at primes above $p$}
A result of Liu \cite[Theorem~0.3.4]{Liu} (see also \cite[\S3.3.2]{BenoisBuyukbodukCKPRVolume} for a detailed summary) shows that the $(\varphi,\Gamma_\cyc)$-module $\DdagrigA (V_{\f})$ admits a triangulation over $\mathscr{A}$. In more precise terms, the module $\DdagrigA (V_{\f})$ sits in an exact sequence
\begin{equation}
\label{eqn:triangulationoverA}
0\rightarrow \mathbb D_{\ubeta} \rightarrow \DdagrigA (V_{\f})
\rightarrow \widetilde{\mathbb D}_{\ubeta}\rightarrow 0,
\end{equation}
where both $\mathbb D_{\ubeta}$ and $\widetilde{\mathbb D}_{\ubeta}$ are $(\varphi,\Gamma_\cyc)$-modules of rank $1$, and where $\mathbb D_{\ubeta}$ is the $(\varphi,\Gamma_\cyc)$-module of rank one associated to the character $\delta_{\ubeta}:\QQ_p^\times\to \mathscr{A}^\times$ given by $\delta_{\ubeta}(p)=\ubeta$ and $\delta_{\ubeta}\,{\vert_{\ZZ_p^\times}}=1$.

Recall that we have assumed $p=\p\p^c$ splits, so that $K_\q=\QQ_p$ for each $\q \in \{\p,\p^c\}$. We define $U_\q^+ (V_{\f},\mathbb D_{\ubeta}):=C^{\bullet}_{\varphi,\gamma}(\mathbb D_{\ubeta})$. On composing the quasi-isomorphism $\eta$ of Proposition~\ref{proposition:quasi-isomorphisms}  with the canonical morphism $U_\q^+ (V_{\f},\mathbb D_{\ubeta}) \rightarrow C^{\bullet}_{\varphi,\gamma}(\DdagrigA (V))$, we obtain a map 
\[
i_\q^+\,:\,U_\q^+ (V_{\f},\mathbb D_{\ubeta}) \lra K^{\bullet}_\q(V_{\f})
\]
where $K^{\bullet}_\q(V_{\f}):=\mathrm{Tot} \left (C^{\bullet}\left(G_\q,V_{\f}\otimes_{\mathscr{A}}\widetilde{\mathbb{B}}_{\rm{rig},\mathscr{A}}^{\dagger}\right) \xrightarrow{\varphi-1} C^{\bullet}\left(G_\q,V_{\f}\otimes_{\mathscr{A}}\widetilde{\mathbb{B}}_{\rm{rig},\mathscr{A}}^{\dagger}\right)\right)$ as above.

\subsubsection{Local conditions away from $p$}
For each non-archimedean prime $\lambda\in \Sigma\setminus \{\p,\p^c\}$ of $K$, we define the complex 
\[
U_{\lambda}^+(V)= \left [V^{I_{\lambda}}_{\f} \xrightarrow{{\rm Fr}_{\lambda}-1} V_{\f}^{I_{\lambda}}\right ],
\]
which is concentrated in degrees $0$ and $1$ and where ${\rm Fr}_{\lambda}$ denotes the geometric Frobenius.
We define
\[
i_{\lambda}^+\,:\,U_{\lambda}^+(V_{\f}) \lra C^{\bullet}(G_{\lambda},V_{\f})
\]
by setting
\[
\begin{aligned}
&i_{\lambda}^+(x)= x &&\text{in degree $0$,}\\
&i_{\lambda}^+(x)({\rm Fr}_{\lambda})=x &&\text{in degree $1$.}
\end{aligned}
\]
In order to have a uniform notation for all primes in $\Sigma$ we set $K_{\lambda}^{\bullet}(V_{\f}):=C^{\bullet}(G_{\lambda},V_{\f})$ and $U_{\lambda}^+(V_{\f},\mathbb D_{\ubeta}):=U_{\lambda}^+(V_{\f})$ for a non-archimedean prime $\lambda\in \Sigma\setminus \{\p,\p^c\}$. Since we assume $p>2$, we may safely ignore the archimedean places. 
\subsubsection{The Selmer complex}
\label{subsubsec:selmercomplexbig}
We define  the complexes $K_\Sigma^{\bullet}(V_{\f}):=\underset{\lambda\in \Sigma}\bigoplus K_{\lambda}^\bullet(V_{\f})$
and $U^+_\Sigma (V_{\f},\mathbb D_{\ubeta}):=\underset{\lambda\in \Sigma}\bigoplus U_{\lambda}^+(V_{\f},\mathbb{D}_{\ubeta}).$  Observe that we have a diagram
\[
\xymatrix{
C^{\bullet}(G_{K,\Sigma},V_{\f}) \ar[r]^(.55){\res_\Sigma} &K_\Sigma^{\bullet}(V_{\f})\\
& U_\Sigma^+ (V_{\f},\mathbb D_{\ubeta}) \ar[u]^{i_\Sigma^+},}
\]
where $i_\Sigma^+=(i_{\lambda}^+)_{\lambda\in \Sigma}$ and $\res_\Sigma$ denotes the localization map. 
\begin{defn}
The Selmer complex associated to these data is defined as 
\[
S^{\bullet}(V_{\f_{/K}}, \mathbb D_{\ubeta})=\mathrm{cone} \left (C^{\bullet}(G_{K,\Sigma},V_{\f}) \oplus 
U_\Sigma^+(V_{\f},\mathbb D_{\ubeta}) \xrightarrow{\res_\Sigma-i_\Sigma^+} K_\Sigma^{\bullet}(V_{\f}) \right)[-1].
\]
 \end{defn}
\begin{defn}
We denote by  $\mathbf{R}\Gamma  (V_{\f_{/K}},\mathbb D_{\ubeta})$ the class of $S^{\bullet} (V_{\f_{/K}},\mathbb D_{\ubeta})$
in the  derived category of $\mathscr{A}$-modules  and denote by  
\[
H^i (V_{\f_{/K}},\mathbb D_{\ubeta}):= \mathbf{R}^i\Gamma (V_{\f_{/K}},\mathbb D_{\ubeta}).
\]
its cohomology. 
\end{defn}

\subsection{$\mathscr{A}$-adic cyclotomic height pairings}
\label{subsec:Aadicheights}
We provide in this section an overview of the construction of $p$-adic heights for $p$-adic representations 
over the affinoid algebra $\mathscr{A}$, following \cite{Ben14b}. We retain our previous notation and conventions.  

Let $J_{\mathscr{A}}$ denote the kernel of the augmentation map 
$$\mathscr{H}\,\widehat{\otimes}\,\scrA=:\mathscr R_{+,\scrA}\rightarrow \scrA$$
which is induced by $\gamma\mapsto 1$. Note that $J_{\scrA}=(\gamma-1)\mathscr R_{+,\scrA}$ and $J_{\scrA}/J_{\scrA}^2\simeq \scrA$  as $\scrA$-modules. The exact sequence 
\[
0\rightarrow V_{\f}\otimes J_{\scrA}/J_{\scrA}^2\lra V_{\f}\otimes  \mathscr{R}_{+,\scrA}/J_{\scrA}^2 \lra V_{\f}  \lra 0
\]
and the functorial behaviour of Selmer complexes under base change induces the Bockstein morphism
\[
\beta^{\mathrm{cyc}}_{V_{\f},\mathbb D_{\ubeta}}: \mathbf{R}\Gamma (V_{\f_{/K}},\mathbb D_{\ubeta}) \lra \mathbf{R}\Gamma  (V_{\f_{/K}},\mathbb D_{\ubeta})[1]\otimes_\scrA J_{\scrA}/J_{\scrA}^2.
\]

\begin{defn}
\label{definition:height pairing} The $p$-adic height pairing associated to the Coleman family $(\f,\ubeta)$ is defined as the morphism
\begin{multline*}
\mathfrak{h}_{\f,\ubeta}: \,
\mathbf{R}\Gamma(V_{\f_{/K}},\mathbb D_{\ubeta})\otimes_\scrA^{\mathbb L}\mathbf{R}\Gamma(V_{\f_{/K}},\mathbb D_{\ubeta}) \xrightarrow{\beta^{\mathrm{cyc}}_{V_{\f},\mathbb D_{\ubeta}}\,\otimes\,\mathrm{id}}\\
\left(\mathbf{R}\Gamma(V_{\f_{/K}},\mathbb D_{\ubeta})[1]\otimes J_{\scrA}/J_{\scrA}^2\right)\otimes_\mathscr{A}^{\mathbb L}\mathbf{R}\Gamma (V_{\f_{/K}},\mathbb D_{\ubeta})
\xrightarrow{\cup} J_{\scrA}/J_{\scrA}^2[-2]
\end{multline*}
where $\cup$ is the cup-product pairing 
$$\mathbf{R}\Gamma(V_{\f_{/K}},\mathbb D_{\ubeta})\otimes_\scrA^{\mathbb L}\mathbf{R}\Gamma(V_{\f_{/K}},\mathbb D_{\ubeta})\stackrel{\cup}{\lra}\mathscr{A}[-3]$$
which is induced from the  $G_{K,\Sigma}$-equivariant  symplectic pairing $\langle\,,\,\rangle_{Np^\infty}$ of \eqref{pairingforVf}.
\end{defn}

In the level of cohomology, $h_{\f,\ubeta}$ induces a pairing  
\[
\mathfrak{h}_{\f,\ubeta}^{1,1}:\, H^1(V_{\f_{/K}},\mathbb D_{\ubeta})\otimes_{\scrA} H^1(V_{\f_{/K}},\mathbb D_{\ubeta})\lra J_\scrA/J_\scrA^2.
\]

\begin{proposition}
\label{prop:symmetricpropertyofheights} 
The $\mathscr{A}$-adic height pairing $\mathfrak{h}_{\f,\ubeta}^{1,1}$ is symmetric.
\end{proposition}
\begin{proof} This is a direct consequence of \cite[Theorem~I]{Ben14b} and the fact that the pairing $\langle\,,\,\rangle_{Np^\infty}$ is symplectic.
\end{proof}
The map $\gamma-1\pmod{J_\scrA^2}\mapsto \log \chi_\cyc (\gamma)$ induces an isomorphism 
$$\partial_\cyc: J_\scrA/J_\scrA^2\stackrel{\sim}{\lra} \scrA\,.$$
We define the $\scrA$-valued height pairing $\mathfrak{h}_{\f,K}$ by setting
$$\mathfrak{h}_{\f,K}:=\partial_\cyc\circ \mathfrak{h}_{\f,\ubeta}^{1,1}\,.$$

\subsection{Specializations and comparison with Nekov\'a\v{r}'s heights}
\label{subsec:specializeheights}

Shrinking $U$ if necessary, we shall assume that $2k\notin I$. Throughout this subsection, we fix an integer $\kappa\in I$ with $\kappa\geq k$ and set 
$$g:=\f(\kappa)\in S_{\kappa}(\Gamma_0(Np)) \hbox{\,\,\, and \,\,\,}b=\ubeta(\kappa).$$
The Galois representation $V_{\f}\otimes_{\mathscr{A},\psi_\kappa}E$ is the central-critical twist $V_g$ of Deligne's representation $W_g$ associated to the cuspidal eigenform $g$. 

\begin{lemma}\label{lemma_old_non_theta_critical}
The eigenform $g$ is non-$\theta$-critical and old at $p$.
\end{lemma}
\begin{proof}
If $\kappa>k$, the eigenform for $g$ is not critical since in this case we have $v_p(b)=k-1<\kappa-1$. If $\kappa=k$, then $g=f^\beta$ is non-$\theta$-critical by assumption.

If $g$ were new at $p$, we would have $k-1=v_p(b)=\kappa/2-1$ and thus $\kappa=2k$, contradicting the choice of $I$.
\end{proof}

\begin{corollary}
\label{cor:giscrytalline}
The Galois representation $V_g$ is crystalline at $p$.
\end{corollary}

\begin{defn}
\label{def:unstabilizedform}
We let $g^\circ\in S_\kappa(\Gamma_0(N))$ denote the newform such that $g=(g^\circ)^b$ is the $p$-stabilization of $g^\circ$.
\end{defn}

Consider the Bloch-Kato Selmer group $H^1_{\rm f}(K,V_g)$. It comes equipped with Nekov\'a\v{r}'s $p$-adic height pairing
$$h^{\rm Nek}_{b,K}:\,H^1_{\rm f}(K,V_g) \otimes H^1_{\rm f}(K,V_g) \lra E.$$
The height pairing $h^{\rm Nek}_{b,K}$ is associated to the Hodge-splitting
$${\rm D}_{\rm cris}(V_g)={\rm D}_b\oplus {\rm Fil}^0{\rm D}_{\rm cris}(V_g)$$
together with the symplectic pairing 
$$\langle\,,\,\rangle_{N}\,:\, V_g \otimes V_g \lra E(1)$$
that is induced from the Poincar\'e duality for the \'etale cohomology of the modular curve $X_0(N)$, where $V_g=V_{g^\circ}$ appears as a direct summand.

Our goal in this subsection is to compare these objects to those obtained by specializing the $\scrA$-adic objects we have defined in the previous section. 

\begin{defn}
\label{def:Poincaredualitypairings}
\item[i)] We let $W_{Np}$ denote the Atkin-Lehner operator of level $Np$ and let $\langle\,,\, \rangle_{Np}$ denote the Poincar\'e duality pairing on the cohomology of the modular curve of level $\Gamma_0(N)\cap \Gamma_1(p)$. 
\item[ii)] Realizing $V_g$ as the $g$-isotypical $($with respect to the dual Hecke operators $T_\ell^\prime$ for $\ell \nmid Np$ and operators $U_\ell^\prime$ for $\ell\mid Np$$)$ quotient of the cohomology of the modular curve of level $\Gamma_0(N)\cap \Gamma_1(p)$ with coefficients in ${\rm TSym}^{\kappa-2}(\mathscr{H}_{\ZZ_p})$ on the left and as the $g$-isotypical $($with respect to the Hecke operators $T_\ell$ for $\ell \nmid Np$ and operators $U_\ell$ for $\ell\mid Np$$)$ direct summand of the modular curve of level $\Gamma_0(N)\cap \Gamma_1(p)$ with coefficients in ${\rm Sym}^{\kappa-2}(\mathscr{H}_{\ZZ_p}^{\vee})$ on the right, we define
$$\langle \,,\, \rangle_{Np}^\prime\,:\,V_g \otimes V_g \lra E(1)$$
 by setting 
$$\langle x,y\rangle_{Np}^\prime:=\langle x,W_{Np}^{-1}\,y \rangle_{Np}$$ 
and refer to it as the $p$-stabilized Poincar\'e duality pairing on $V_g$. Here, we follow the notation in \cite[\S2.3]{KLZ2} to define the sheaf $\mathscr{H}_{\ZZ_p}$.
\item[iii)] We let 
$${\rm Pr}_b^*: V_g \stackrel{\sim}{\lra} V_g$$ 
denote the natural isomorphism of Galois representations appearing in \cite[Proposition 10.1.1/1]{KLZ2}\footnote{This map would have been denoted by $({\rm Pr}^b)^*$ in op. cit.}. Here $V_g$ on the left is the $g$-isotypical direct summand in the cohomology of $X_0(N)$ $($with respect to the Hecke operators $T_\ell$ for $\ell \nmid N$ and operators $U_\ell$ for $\ell\mid N$$)$, whereas $V_g$ on the right is the $g$-isotypical direct summand in the cohomology of the modular curve of level $\Gamma_0(N)\cap \Gamma_1(p)$ $($with respect to the Hecke operators $T_\ell$ for $\ell \nmid Np$ and operators $U_\ell$ for $\ell\mid Np$$)$. We also have an isomorphism $({\rm Pr}_b)_*: V_g \stackrel{\sim}{\lra} V_g$ in the reverse direction.
\end{defn}
\begin{remark}
\label{rem_compare_V_g_V_g_star_Atkin}
In this remark, let us write $V_g^*$ for the (self-dual twist of the) $g$-isotypical direct summand of the modular curve of level $\Gamma_0(N)$ with coefficients in ${\rm Sym}^{\kappa-2}(\mathscr{H}_{\QQ_p}^{\vee})$. Also in this remark, we let $V_g$ denote the (self-dual twist of the) $g$-isotypical quotient of the modular curve of level $\Gamma_0(N)$ with coefficients in ${\rm TSym}^{\kappa-2}(\mathscr{H}_{\QQ_p})$. The Atkin--Lehner involution induces an isomorphism 
$$\lambda_N(g^\circ)^{-1}W_N: V_g^*\stackrel{\sim}{\lra} V_g\,.$$
where $\lambda_N(g^\circ)$ is the Atkin-Lehner pseudo-eigenvalue of $g^\circ$. The Heegner classes (denoted by $z_g^\circ\in H^1(K,V_g^*)$ and introduced in \S\ref{subsubsec:HeegnerCycles} below) naturally take coefficients in the representation $V_g^*$.  The identification via this involution will be made implicit by not differentiating $V_g$ and $V_g^*$. We will then write, except in  Theorem~\ref{thm:bigheegnermain} and Corollary~\ref{cor:padicGZoverKThatInterpolates}, $z_g^\circ\in H^1(K,V_g)$ in place of writing $\lambda_N(g^\circ)^{-1}W_N(z_g^\circ)$. This matter will only be crucial in Theorem~\ref{thm:bigheegnermain} and Corollary~\ref{cor:padicGZoverKThatInterpolates}, where the interpolative properties of Generalized Heegner cycles require to keep track of the space where the Galois representation $V_g$ is realized. When we write (by an abuse of notation) $\lambda_N(g)^{-1}W_N: V_g \lra V_g$, we will understand that $V_g$ on the left is a direct summand of $H^{1}_{\textup{\'{e}t}}(Y_1(N)_{\overline{\QQ}},{\rm Sym}^{\kappa-2}({\mathscr{H}_{\QQ_p}^\vee}))(\kappa/2)$, whereas $V_g$ on the right is a direct summand of $H^{1}_{\textup{\'{e}t}}(Y_1(N)_{\overline{\QQ}},{\rm TSym}^{\kappa-2}({\mathscr{H}_{\QQ_p}}))(1-\kappa/2)$.
\end{remark}
See \cite[\S3.1.1]{BLGHCpatching} for an elaboration on the self-duality pairings which make an appearance in Definition~\ref{def:Poincaredualitypairings}; most particularly, for the appearance of the Atkin--Lehner operator. We are grateful to D. Loeffler for bringing the following observation to our attention.

\begin{proposition}
\label{prop:comparisionoftwopoincarepairings}
\item[i)] $\langle \psi_\kappa\, x, \psi_\kappa\, y\rangle_{Np}^\prime=\psi_\kappa\circ \langle x,y\rangle_{Np^\infty}$.
\item[ii)]$\langle X\,,\, W_{Np}{\rm Pr}_b^*\, y\rangle_{Np}^\prime=\langle  ({\rm Pr}_b)_*X\,,\,y\rangle_{N}$.
\item[iii)] We have 
$$\langle{\rm Pr}_b^*\, x\,,\, {\rm Pr}_b^*\, y\rangle_{Np}^\prime= b \,\lambda_N(g^\circ)\, \mathcal{E}(g)  \mathcal{E}^* (g)\, \langle x, y\rangle_{N},$$
where $\mathcal{E}(g)=\left(1-\frac{p^{\kappa-2}}{b^2}\right)$ and $\mathcal{E}^*(g)=\left(1-\frac{p^{\kappa-1}}{b^2}\right)$.
\end{proposition}

\begin{proof}
The first assertion is well-known; c.f. Proposition 4.4.8 and Theorem 4.6.6 of \cite{LZ1}. The second identity follows from definitions; c.f. \cite[Diagram (5)]{BLGHCpatching}. For the third, we note that 
\begin{align*}
\langle{\rm Pr}_b^*\, x\,,\, {\rm Pr}_b^*\, y\rangle_{Np}^\prime&= \langle {\rm Pr}_b^*\,x\,,\, W_{Np}\,{\rm Pr}_b^*\, y\rangle_{Np}\\
&=\langle x\,,\,({\rm Pr}_b)_*\, W_{Np}\,{\rm Pr}_b^*\, y\rangle_{N}\\
&=b \,\lambda_N(g^\circ)\, \mathcal{E}(g)  \mathcal{E}^* (g)\, \langle x, y\rangle_{N}
\end{align*}
where the first and second equalities follow from definitions, whereas the third is a consequence of the discussion in the final paragraph of the proof of Proposition 10.1.1 of \cite{KLZ2}.
\end{proof}
Since the $g$ is non-$\theta$-critical (Lemma~\ref{lemma_old_non_theta_critical}), the triangulation \eqref{eqn:triangulationoverA} gives rise to a saturated triangulation 
$$0\lra \DD_b\lra \Ddagrig(V_g)\lra \widetilde{\DD}_b\lra 0$$
of the $(\varphi,\Gamma_\cyc)$-module $\Ddagrig(V_g)$ by base change, where $\DD_b:=\mathbb D_\ubeta\otimes_{\scrA,\psi_\kappa}E$ and $\widetilde \DD_b:=\widetilde{\mathbb D}_\ubeta\otimes_{\scrA,\psi_\kappa}E$. With this data at hand, one may proceed precisely as in \S\ref{subsubsec:selmercomplexbig} to define a Selmer complex $S^\bullet(V_{g_{/K}},\DD_b)$ in the category of $E$-vector spaces. We let $\mathbf{R}\Gamma(V_{g_{/K}},\DD_b)$ denote the corresponding object in the derived category and $H^i(V_{g_{/K}},\DD_b)$ denote its cohomology. 

The general formalism to construct $p$-adic heights we outlined in \S\ref{subsec:Aadicheights} (where we utilize the symplectic pairing 
$$\langle \,,\, \rangle_{Np}^\prime\,:\, V_g\otimes V_g \lra E(1)$$
given in Definition~\ref{def:Poincaredualitypairings} to determine an isomorphism $V_g^*(1)\stackrel{\sim}{\ra} V_g$) also equips us with an $E$-valued height pairing 
$$\mathfrak{h}_{g,b,K}\,:\,H^1(V_{g_{/K}},\DD_b)\otimes H^1(V_{g_{/K}},\DD_b)\lra E\,.$$

\begin{lemma}
\label{lemma:specializedselmergroupversusthebase}
\item[{\rm{i)}}] We have a natural morphism $($which we shall denote by $\psi_\kappa$, by slight abuse$)$
$$\psi_\kappa:\,H^1(V_{\f_{/K}},\mathbb D_{\ubeta})\otimes_{\scrA,\psi_\kappa}E{\lra} H^1(V_{g_{/K}},\DD_b)$$
of $E$-vector spaces, which is an isomorphism for all but finitely many choices of $g$. 
\item[{\rm{ii)}}] The following diagram commutes:
$$\xymatrix@C=.1cm{H^1(V_{\f_{/K}},\mathbb D_{\ubeta})\ar[d]_{\psi_\kappa}&\otimes_{\scrA}& H^1(V_{\f_{/K}},\mathbb D_{\ubeta})\ar[d]_{\psi_\kappa}\ar[rrrrr]^(.65){\mathfrak{h}_{\f,K}}&&&&&\mathscr{A}\ar[d]_{\psi_\kappa}\\
 H^1(V_{g_{/K}},\DD_b)&\otimes_E& H^1(V_{g_{/K}},\DD_b)\ar[rrrrr]_(.65){\mathfrak{h}_{g,b,K}}&&&&& E
}$$

\end{lemma}
\begin{proof}
Let $\wp_\kappa:=\ker(\psi_\kappa)$ be the prime of $\mathscr{A}$ corresponding to $g$. Notice then that 
$$H^1(V_{\f_{/K}},\mathbb D_{\ubeta})\otimes_{\scrA,\psi_\kappa}E=H^1(V_{\f_{/K}},\mathbb D_{\ubeta})/\wp_\kappa H^1(V_{\f_{/K}},\mathbb D_{\ubeta})$$ 
and the general base change principles for Selmer complexes (c.f. \cite[Section 1]{Pot13}) show that the sequence
$$0\lra H^1(V_{\f_{/K}},\mathbb D_{\ubeta})/\wp_\kappa H^1(V_{\f_{/K}},\mathbb D_{\ubeta})\lra H^1(V_{g_{/K}},\DD_b)\lra  H^2(V_{\f_{/K}},\mathbb D_{\ubeta})[\wp_\kappa]$$ 
of $E$-vector spaces is exact. The first assertion now follows. The second follows easily from definitions. 
\end{proof}

\begin{proposition}
\label{prop:comparisionofheightswithouttrivialzeros}
There is a natural isomorphism 
$$H^1(V_{g_{/K}},\DD_b) \stackrel{\sim}{\lra} H^1_{\rm f}(K,V_g)\,.$$
Moreover, the height pairing $\mathfrak{h}_{g,b,K}$ coincides with ${\displaystyle {h^{\rm Nek}_{b,K}}\big{/}{b \,\lambda_N(g^\circ)\, \mathcal{E}(g)  \mathcal{E}^* (g)}}$.
\end{proposition}
\begin{proof} The proof of the first assertion reduces to \cite[Theorem~III]{Ben14b} once we verify 
\begin{itemize}
\item[(i)] ${\rm D}_{\rm cris}(V_g)^{\varphi=1}=0$,
\item[(ii)] $H^0(\widetilde{\DD}_b)=0$.
\end{itemize}
Assume first $\kappa\neq k$ (so that $g\neq f^\beta$). Let $g^\circ$ be as in Definition~\ref{def:unstabilizedform}. The roots of the Hecke polynomial for $g^\circ$ at $p$ could not be the pair $\{1,p^{\kappa-1}\}$, as otherwise we would have $\kappa-1=v_p(b)=k-1$. This verifies both conditions in this case.

When $\kappa=k$ and $g=f^\beta$, both conditions follow as a consequence of the Ramanujan-Petersson conjecture for $f$ (as proved by Deligne), according to which the roots of the Hecke polynomial of $f$ at $p$ could not be the pair $\{1,p^{k-1}\}$.

The assertion concerning the comparison of two $p$-adic heights follows from \cite[Theorem~11]{Ben14b} together with Proposition~\ref{prop:comparisionoftwopoincarepairings}(iii). (We find it instructive to compare Benois' result to  \cite[Theorem~11.4.6]{Ne06} in the ordinary case.)
\end{proof}
The following commutative diagrams summarize the discussion in this subsection:
\begin{equation}
\label{eqn:comaprisonofheights}
\xymatrix@C=.1cm{H^1(V_{\f_{/K}},\mathbb D_{\ubeta})\ar[d]_{\psi_\kappa}&\otimes_{\scrA}& H^1(V_{\f_{/K}},\mathbb D_{\ubeta})\ar[d]_{\psi_\kappa}\ar[rrrrr]^(.65){\mathfrak{h}_{\f,K}}&&&&&\mathscr{A}\ar[d]_{\psi_\kappa}\\
 H^1(V_{g_{/K}},\DD_b)&\otimes_E& H^1(V_{g_{/K}},\DD_b)\ar[rrrrr]_(.65){\mathfrak{h}_{g,b,K}}&&&&& E\\
 H^1_{\rm f}(K,V_g)\ar[u]^{\cong}_{{\rm Pr}_b^*}&\otimes_E& H^1_{\rm f}(K,V_g)\ar[u]^{\cong}_{{\rm Pr}_b^*}\ar[rrrrr]_(.65){{h}_{b,K}^{\rm Nek}}&&&&& E \ar[u]_{b \,\lambda_N(g^\circ)\, \mathcal{E}(g)  \mathcal{E}^* (g)}
}\end{equation}

The following diagram is more useful for our purposes. It is obtained by using Proposition~\ref{prop:comparisionoftwopoincarepairings}(ii); see also \cite[Diagram (5)]{BLGHCpatching} for a detailed discussion that concerns the lower half of the diagram \eqref{eqn:comaprisonofheights_2}.
\begin{equation}
\label{eqn:comaprisonofheights_2}
\xymatrix@C=.1cm{H^1(V_{\f_{/K}},\mathbb D_{\ubeta})\ar[d]_{\psi_\kappa}&\otimes_{\scrA}& H^1(V_{\f_{/K}},\mathbb D_{\ubeta})\ar[d]_{\psi_\kappa}\ar[rrrrr]^(.65){\mathfrak{h}_{\f,K}}&&&&&\mathscr{A}\ar[d]_{\psi_\kappa}\\
 H^1(V_{g_{/K}},\DD_b)\ar[d]^{\cong}_{({\rm Pr}_b)_*}&\otimes_E& H^1(V_{g_{/K}},\DD_b)\ar[rrrrr]_(.65){\mathfrak{h}_{g,b,K}}&&&&& E\\
 H^1_{\rm f}(K,V_g)&\otimes_E& H^1_{\rm f}(K,V_g)\ar[u]^{\cong}_{W_{Np}\circ\,{\rm Pr}_b^*}\ar[rrrrr]_(.65){{h}_{b,K}^{\rm Nek}}&&&&& E\ar@{=}[u]
}\end{equation}
\subsection{Universal Heegner points}
\label{subsec:BigHeegstatementsconsequences}
In this subsection, we shall introduce elements in the Selmer groups on which we shall calculate the $\mathscr{A}$-adic height $\mathfrak{h}_{\f,K}$.
\subsubsection{Heegner cycles}
\label{subsubsec:HeegnerCycles}
 We recall the definition of Heegner cycles on Kuga-Sato varieties, following the discussion in \cite{NekovarGZ}. Recall that we have fixed an imaginary quadratic field $K$ such that all primes dividing the tame level $Np$ split completely in $K/\QQ$. Let $g\in S_\kappa(\Gamma_0(N))$ be a cuspidal eigenform of weight $\kappa> 2$.

Let $Y(N)$ denote the modular curve over $\QQ$ which is the moduli of elliptic curves with full level $N$ structure and we let $\mathfrak{j}:\, Y(N) \ra X(N)$ denote its non-singular compactification. Since we assume $N \geq 3$, there is a universal generalized elliptic curve $E\ra X(N)$ that restricts to the universal elliptic curve
$f :\, E\ra Y (N)$. The $(\kappa- 2)$-fold fibre product of $E$ with itself over $Y(N)$ has a canonical non-singular compactification $W$ described in detail in \cite{Deligne1971Bourbaki, Scholl1990Inventiones}. We have natural maps
\begin{equation}
\label{eqn:KugaSatotoVg}
H^{\kappa-1}_{\textup{\'{e}t}}(W\times_{\QQ}\overline{\QQ},\QQ_p)(\kappa/2)\ra H^{1}_{\textup{\'{e}t}}(X(N)\times_{\QQ}\overline{\QQ},\mathfrak{j}_*{\rm Sym}^{\kappa-2}({\mathscr{H}_{\QQ_p}^\vee}))(\kappa/2)\ra V_g\,.
\end{equation}
where $\mathscr{H}_{\QQ_p}:=(R^1f_*\QQ_p)^\vee=R^1f_*\QQ_p(1)$. Scholl defines a projector $\varepsilon$ (where his $w$ corresponds to our $\kappa-2$) and proves that there is a canonical isomorphism
$$H^{1}_{\textup{\'{e}t}}(X(N)\times_{\QQ}\overline{\QQ},\mathfrak{j}_*{\rm Sym}^{\kappa-2}({\mathscr{H}_{\QQ_p}^\vee}))\stackrel{\sim}{\lra}\varepsilon H^{\kappa-1}_{\textup{\'{e}t}}(W\times_{\QQ}\overline{\QQ},\QQ_p)\,.$$ 
We finally define 
$$\mathscr{B}:=\left\{\left(\begin{array}{cc}* & *\\ 0& * \end{array} \right) \right\}\Big{/}\{\pm 1\} \subset \GL_2(\ZZ/N\ZZ)\big{/}\{\pm 1\} $$
and the idempotent $\varepsilon_{\mathscr B}:=\frac{1}{|\mathscr{B}|}\sum_{g\in \mathscr{B}}g$ (which acts on the modular curves $Y(N)$ and $X(N)$). 
\begin{defn}
We let $\mathfrak{N}$ be an ideal of $\cO$ such that $\cO/\mathfrak{N}\cong\ZZ/N\ZZ$. For an arbitrarily chosen ideal $\mathcal{A}\subset \cO$, consider the isogeny $\mathbb{C}/\mathcal{A}\ra \mathbb{C}/\mathcal{A}{\mathfrak N}^{-1}$. It represents the Heegner point $y=y_{\mathcal A}$ on $Y_0(N)(\mathbb{C})$. It is defined over the Hilbert class field $H$ of $K$.  
\end{defn}
Following \cite{NekovarGZ}, we choose a point $\widetilde{y}\in Y(N)\times_{\QQ}H$ defined over $H$, lying above the Heegner point $y$ (viewed as a closed point of $Y_0(N)\times_\QQ H$). The fiber $E_{\widetilde{y}}$ is a CM elliptic curve defined over $H$ whose endomorphism ring is isomorphic to $\cO$. We let 
$$\Gamma_{\sqrt{{-D_K}}}\subset E_{\widetilde{y}}\times E_{\widetilde{y}}$$ 
denote the graph of $\sqrt{{-D_K}}\in \cO$ (fix any one of the two square-roots)
\begin{defn}
We let 
$$Y:=\underbrace{\Gamma_{\sqrt{{-D_K}}}\times \cdots \times\Gamma_{\sqrt{{-D_K}}}}_{\kappa/2-1 \hbox{ times}}\subset \underbrace{E_{\widetilde{y}}\times\cdots\times E_{\widetilde{y}}}_{\kappa-2 \hbox{ times}}=(W\times_\QQ H)_{\widetilde y}$$
and call the cycle (with rational coefficients) that is represented by $\varepsilon_{\mathscr B}\varepsilon Y$
inside of $\varepsilon_{\mathscr B}\varepsilon {\rm CH}^{\kappa/2}(W\times_\QQ H)_0\otimes \QQ$ (which we also denote by the same symbol $\varepsilon_{\mathscr B}\varepsilon Y$) the Heegner cycle. 
\end{defn}

The cohomology class of $\varepsilon Y$ in $H^{\kappa}_{\textup{\'{e}t}}(W\times_{\QQ}\overline{\QQ},\QQ_p)(\kappa/2)$ vanishes, so that one may apply the Abel-Jacobi map
$${\rm AJ}:{\rm CH}^{\kappa/2}(W\times_{\QQ}H)_0\otimes \QQ\lra H^1(H,H^{\kappa-1}_{\textup{\'{e}t}}(W\times_{\QQ}\overline{\QQ},\QQ_p)(\kappa/2))$$
on the Heegner cycle $\varepsilon_{\mathscr B}\varepsilon Y$. 
\begin{defn}
We let ${\rm AJ}_g:\,{\rm CH}^{\kappa/2}(W\times_{\QQ}H)_0\otimes \QQ \ra H^1(H,V_g)$
denote the compositum of the map \eqref{eqn:KugaSatotoVg} with the Abel-Jacobi map and define the Heegner cycle
$$z_g:={\rm cor}_{H/K}\left({\rm AJ}_g(\varepsilon_{\mathscr B}\varepsilon Y)\right) \in H^1(K,V_g)\,.$$
\end{defn}
\begin{remark}
\label{rem_atkin_lehner_heegner_def}
In view of Remark~\ref{rem_compare_V_g_V_g_star_Atkin}, it is more appropriate to write the map \eqref{eqn:KugaSatotoVg} as the composite
$$
H^{\kappa-1}_{\textup{\'{e}t}}(W\times_{\QQ}\overline{\QQ},\QQ_p)(\kappa/2)\lra H^{1}_{\textup{\'{e}t}}(X(N)\times_{\QQ}\overline{\QQ},\mathfrak{j}_*{\rm Sym}^{\kappa-2}({\mathscr{H}_{\QQ_p}^\vee}))(\kappa/2)\lra V_g^*
$$ 
and $z_g\in H^1(K,V_g^*)$ when we need to be mindful about the space where the Galois representation $V_g$ is realized (which is the case only in Theorem~\ref{thm:bigheegnermain} and Corollary~\ref{cor:padicGZoverKThatInterpolates} below). We retain the notation of Remark~\ref{rem_compare_V_g_V_g_star_Atkin} and write $V_g$ for the $g$-isotypical quotient of $H^{1}_{\textup{\'{e}t}}(X(N)\times_{\QQ}\overline{\QQ},\mathfrak{j}_*{\rm TSym}^{\kappa-2}({\mathscr{H}_{\QQ_p}}))(1-\kappa/2)$. The Atkin--Lehner involution gives rise to an identification 
$V_g^*\stackrel{\sim}{\lra}V_g$ and the Heegner class $\lambda_N(g)^{-1}W_N(z_g)\in H^1(K,V_g)$. We will often not keep track of this involution and write $z_g$ in place of $\lambda_N(g)^{-1}W_N(z_g)$.
\end{remark}
Since $p\nmid N$, all $X(N)$, $X_0(N)$ and $W$ have good reduction at $p$ and it follows from \cite[Theorem 3.1(i)]{Nekovar2000AJHeightsBanff1998} that 
$$z_g\in H^1_{\rm f}(K,V_g)\,.$$
\subsubsection{Heegner cycles in Coleman families}
\label{subsubsec:HeegcycleforColemanfamilies}
For a classical weight $\kappa\in I$ and $\psi_\kappa$ as in \S\ref{subsubsec:colemanfamilyintro},  we let $\f(\kappa)^\circ\in S_{\kappa}(\Gamma_0(N))$ denote the newform whose $p$-stabilization (with respect to $\ubeta(\kappa)$) is the eigenform $\f(\kappa)$. 

The following result (construction of a big Heegner point along the Coleman family $\f$) is \cite[Proposition 4.15]{BLGHCpatching} and \cite[Theorem 5.4.1]{JLZHC}.

\begin{theorem}[B\"uy\"ukboduk--Lei, Jetchev--Loeffler--Zerbes]
\label{thm:bigheegnermain}
There exists a unique class $\mathscr{Z}_{\f}\in H^1(V_{\f_{/K}},\mathbb D_{\ubeta})$ 
that is characterized by the requirements that for any $\kappa\in I$ we have 
\begin{equation}
\label{eqn_bigheegnermain_1}\psi_\kappa\left(\mathscr{Z}_{\f}\right)={\dfrac{\left(1-\dfrac{p^{\frac{\kappa}{2}-1}}{\ubeta(\kappa)}\right)^{2}}{u_K(2\sqrt{-D_K})^{\frac{\kappa}{2}-1}}\frac{W_{Np}\circ({\rm Pr}_{\ubeta(\kappa)})^*({z_{\f(\kappa)^\circ}})}{\ubeta(\kappa)\lambda_N(\f(\kappa)^\circ)\mathcal{E}(\f(\kappa))\mathcal{E}^*(\f(\kappa))} }\in H^1_{\rm f}(K,V_{\f(\kappa)})
\end{equation}
\begin{equation}
\label{eqn_bigheegnermain_2}
({\rm Pr}_{\ubeta(\kappa)})_*\circ\psi_\kappa\left(\mathscr{Z}_{\f}\right)=\dfrac{\left(1-\dfrac{p^{\frac{\kappa}{2}-1}}{\ubeta(\kappa)}\right)^{2}}{u_K(2\sqrt{-D_K})^{\frac{\kappa}{2}-1}}\frac{W_Nz_{\f(\kappa)^\circ}}{\lambda_N(\f(\kappa)^\circ)} \in H^1_{\rm f}(K,V_{\f(\kappa)}),
\end{equation}
where $u_K=|\cO_K^\times|/2$ and $-D_K$ is the discriminant of $K$.
\end{theorem}
We note that the $G_\QQ$-representation $V_{\f(\kappa)}$ that appears on the first line is realized as the $\f(\kappa)$-isotypical quotient of the cohomology of the modular curve of level $\Gamma_0(N)\cap \Gamma_1(p)$ with coefficients in ${\rm TSym}^{\kappa-2}(\mathscr{H}_{\ZZ_p})$, whereas the one that appears on the second line is realized as the $\f(\kappa)^\circ$-isotypical direct summand in the cohomology of the modular curve of level $\Gamma_0(N)$ with coefficients in ${\rm TSym}^{\kappa-2}(\mathscr{H}_{\ZZ_p})$. We also remark that the classical Heegner cycle $z_{\f(\kappa)^\circ}$ descends from the $\f(\kappa)^\circ$-isotypical direct summand in the cohomology of the modular curve of level $\Gamma_0(N)$ with coefficients in ${\rm Sym}^{\kappa-2}(\mathscr{H}_{\ZZ_p}^\vee)$. The Atkin--Lehner involution $\lambda_N(\f(\kappa)^\circ)^{-1}W_N$ which  appears above sends this space to the $\f(\kappa)^\circ$-isotypical direct summand in the cohomology with coefficients in ${\rm TSym}^{\kappa-2}(\mathscr{H}_{\ZZ_p})$. See \cite[\S3.1]{BLGHCpatching} for a detailed elaboration on this rather delicate issue.

\begin{remark}
\label{remark_GHC_proofs}
Jetchev--Loeffler--Zerbes in \cite{JLZHC} rely on the overconvergent \'etale cohomology of Andreatta--Iovita--Stevens. The construction of ``universal'' Heegner cycles in \cite{BLGHCpatching} exploits the $p$-adic construction of rational points, a theme first observed by Rubin~\cite{Rubin92}, and dwells on the formula of Bertolini--Darmon--Prasanna which relates the Bloch--Kato logarithms of {these cycles} to appropriate Rankin--Selberg $p$-adic $L$-values. In~\cite{BPSII}, we will give another construction of ``universal'' Heegner cycles in the context of Emerton's completed cohomology (on realizing the family $\f$ on Emerton's eigensurface).
\end{remark}
%%%%%%%%%%%%%%%%%%%%%%%%%
%%%%%%%%%%%%%%%%%%%%%%%%%%%%%%
%%%%%%%%%%%%%%%%%%%%%%%%%%%%%%
%%%%%%%%%%%%%%%%%%%%%%%%%%%%%%%%%%%

\section{$p$-adic $L$-functions over the imaginary quadratic field $K$}
\label{sec:padicLfunctions}

We introduce the needed $p$-adic $L$-functions for the arguments in this paper.  We first discuss a Rankin-Selberg $p$-adic $L$-function defined over our imaginary quadratic field $K$.  We then compare this $p$-adic $L$-function to a na\"ive product of $p$-adic $L$-functions defined over $\QQ$.

\subsection{Ranking-Selberg $p$-adic $L$-functions}

Loeffler and Zerbes in \cite{LZ1} have constructed $p$-adic $L$-functions (in $3$-variables) associated to families of semi-ordinary Rankin-Selberg products $f_1\otimes f_2$ of eigenforms, where $f_1$ runs through a Coleman family and $f_2$ through a $p$-ordinary family.  (See also \cite{LoefflerpadicRankinSelberg} where the correct interpolation property is extended from all crystalline points to all critical points.)
We shall let $f_2$ vary in a (suitable branch of the) universal CM family associated to $K$ (which we shall recall below), and thus we may reinterpret this $p$-adic $L$-function as a $p$-adic $L$-function associated to the base change of $\f$ to $K$.

Given a Dirichlet character $\eta$ of $p$-power order and conductor, we shall treat it as a character of $\Gamma$ by class field theory (normalized via geometric Frobenius elements). The base change $\eta\circ\mathbb{N}_{K/\QQ}$ is a ray class character of $p$-power order and conductor. Let us denote by ${\rm res}_{K}$ the natural injection $G_K\to G_\QQ$. Then the $p$-adic Galois character associated to the ray class character $\eta\circ\mathbb{N}_{K/\QQ}$ of $p$-power order and conductor is the character $\eta\circ \res_{K}=\eta_{\vert_{G_K}}$.

\subsubsection{CM Hida families}
For a general modulus $\mathfrak{n}$ of $K$, let $K(\mathfrak{n})$ denote the maximal $p$-extension contained in the ray class field modulo $\mathfrak{n}$. We set $H_{\mathfrak{n}}^{(p)}:=\Gal(K(\mathfrak{n})/K)$. In particular, $K(\p^\infty):=\cup K(\p^n)$ is the unique $\ZZ_p$-extension of $K$ which is unramified outside $\p$. We let $\Gamma_\p:=\varprojlim_n H_{\mathfrak{p}^n}^{(p)}$ denote its Galois group over $K$. We fix an arbitrary Hecke character $\psi_0$  of $\infty$-type $(-1,0)$, conductor $\p$ and whose associated $p$-adic Galois character factors through $\Gamma_\p$. Notice then that ${\psi}_0\equiv\mathds{1}_{\p} \mod \m_E$, where we have let $\mathds{1}_{\p}:\left(\cO/\p\right)^\times\ra \cO_E^\times$ denote the trivial character modulo $\p$. We also let $\Gamma_K$ denote the Galois group of the $\Zp^2$-extension of $K$.
\begin{remark}
If the class number of $K$ is prime to $p$, then the Hecke character $\psi_0$ is unique, as the ratio of two such characters would have finite $p$-power order and conductor dividing $\p$. 
\end{remark}

The theta-series 
$$\Theta(\psi_0):=\sum_{(\mathfrak{a}, \p)=1}\psi_0(\mathfrak{a})q^{\NN\mathfrak{a}}\in S_2(\Gamma_1(|D_K| p),\epsilon_K\omega^{-1})$$
  is  a newform and it is the weight two specialization (with trivial wild character) of the CM Hida family $\g$ with tame level $D_K$ and character $\epsilon_K\omega$. The weight one specialization of this CM Hida family with trivial wild character equals the $p$-ordinary theta-series $\Theta^{\ord}(\mathds{1}_\p):=\sum_{(\mathfrak{a},\p)=1}q^{\NN\mathfrak{a}}\in S_{1}(\Gamma_1(D_K p),\epsilon_K)$ of $\mathds{1}_\p$.
\begin{remark}
One may construct the Hida family $\g$ as follows. We let ${\bf{T}}_{|D_K|p}$ denote the Hecke algebra given as in \cite[\S4.1]{LLZ2} and we define the maximal ideal $\mathcal{I}_{\mathfrak{p}}\subset {\bf{T}}_{|D_K|p}$ as in \cite[Definition 5.1.1]{LLZ2}. Note that in order to determine the map $\phi_{\p}$ that appears in this definition, we use the algebraic Hecke character $\psi_0$ we have chosen above. 

It follows by \cite[Prop. 5.1.2]{LLZ2} that $\mathcal{I}_{\mathfrak{p}}$ is non-Eisenstein, $p$-ordinary and $p$-distinguished. By \cite[Theorem~4.3.4]{LLZ2}, the ideal $\mathcal{I}_{\mathfrak{p}}$ corresponds uniquely to a $p$-distinguished maximal ideal $\mathcal{I}$ of the universal ordinary Hecke algebra ${\bf{T}}_{|D_K|p^\infty}$ acting on $H^1_{\ord}(Y_1(|D_K|p^\infty))$ (definitions of these objects may be found in \cite[Definition 4.3.1]{LLZ2}). The said correspondence is induced from Ohta's control theorem~\cite[Theorem 1.5.7(iii)]{ohta99}, which also attaches to $\mathcal{I}_{\p}$ a unique non-Eisenstein, $p$-ordinary and $p$-distinguished maximal ideal $\mathcal{I}_{\p^r}$ of ${\bf{T}}_{|D_K|p^r}$ for each $r\geq 1$ (it is easy to see that it is the kernel of the compositum of the arrows  
$${\bf{T}}_{|D_K|p^r}\stackrel{\phi_{\p^r}}{\lra}\cO_L[H_{\p^r}]\lra\cO_E\lra \cO_E/\varpi_E,$$ and therefore with its original form given in \cite[Definition 5.1.1]{LLZ2}).   The ideal $\mathcal{I}$ determines the CM Hida family $\g$ alluded to above.   
\end{remark}

Let us denote by $\mathscr{W}(\Gamma_\p)$ the rigid analytic space over $\Qp$ parametrizing continuous characters of $\Gamma_\p$, so that whenever $A$ is an affinoid $\Qp$-algebra, we have $\mathscr{W}(\Gamma_\p)(A)={\rm Hom}_{\rm cts}(\Gamma_\p,A^\times)$. We may and will think of $\mathscr{W}(\Gamma_\p)(\mathbb{C}_p)$ as the open unit ball.

We shall henceforth identify $\mathscr{W}(\Gamma_\p)$ with a finite flat cover of the  weight space $\mathscr{W}(\Gamma)$ for the Hida family $\g$, whose $\mathbb{C}_p$-points are given by
$$\mathscr{W}(\Gamma)(\mathbb{C}_p)={\rm Hom}_{\rm cts}(\Gamma,\mathbb{C}_p^\times)\xrightarrow[\langle\chi_\cyc\rangle]{\sim} {\rm Hom}_{\rm cts}(1+p\Zp,\mathbb{C}_p^\times)$$
(c.f. \cite{ghatevatsal2004AIF}, p. 2157--2158). We let $\widetilde{\kappa}\in \mathscr{W}(\Gamma_\p)(\mathbb{\QQ}_p)$ denote the point corresponding to the weight one specialization $\Theta^{\rm ord}(\mathds{1}_\wp)$. Given a Dirichlet character of order $p^r$ and $p$-power conductor (so that its $p$-adic Galois character, which we have also denoted by $\eta$, factors through $\Gamma$), we denote by $\eta^{(\p)}\in {\rm Hom}_{\rm cts}(\Gamma_\p,\mathbb{C}_p^\times)$ any lift of the Galois character $\eta$. Observe then that the following diagram commutes for each such choice:
$$\xymatrix@C=0.1cm{\Gamma_K\ar[rd]_{\eta_{\vert_{G_K}}}\ar@{->>}[rr]&& \Gamma_\p\ar[ld]^{\eta^{(\p)}}\\
&\mathbb{C}_p^\times&
}$$
Then we have $\g(\eta^{(\p)})\in S_1(\Gamma_1(D_Kp^{r+1}),\epsilon_K\eta)$ for the specialization of the Hida family at the point corresponding to $\eta^{(\p)}$.

\subsubsection{The $p$-adic $L$-function and interpolation property}
We fix a neighborhood $\mathscr{Y}=\overline{B}(\widetilde{\kappa},p^{-\alpha})\subset {\rm Spm}\,\mathscr{W}(\Gamma_\p)$ about $\widetilde{\kappa}$ of radius $p^{-\alpha}$ for some positive rational number $\alpha$ (recall our identification of $\mathscr{W}(\Gamma_\p)$ with the rigid analytic unit ball). The neighborhood $\mathscr{Y}$ has a structure of an affinoid space over $\QQ_p$ and we denote by $\cO(\mathscr Y)$ the ring of analytic functions on $\mathscr Y$ (which is isomorphic to a Tate algebra). Let
$$L_p^{\RS}(\f,\g\vert_{\mathscr Y})\in \mathscr{O(Y)}\,\widehat{\otimes}\,\mathscr{H}_{\mathscr{A}} $$ 
denote the $3$-variable Rankin-Selberg $p$-adic $L$-function of Loeffler and Zerbes \cite{LZ1,LoefflerpadicRankinSelberg}. Since $\g$ is a $p$-ordinary family, we may choose $\mathscr{Y}$ as large as we like, e.g. (thanks to Rohrlich's generic non-vanishing results) to contain a point
$$\eta^{(\p)}\in \mathscr{W}(\Gamma_\p)(\mathbb{C}_p)$$
such that $L(f_{/K},\eta^{-1}\circ\mathbb{N}_{K/\QQ},\frac{k}{2})\neq 0$. In fact, on passing to limit in $\mathscr{Y}$, we obtain a $p$-adic $L$-function
$$L_p^{\RS}(\f,\g)\in \mathscr{H}(\Gamma_\p)\,\widehat{\otimes}\,\mathscr{H}_{\mathscr{A}}$$
where the completed tensor product of two Fr\'echet spaces is given as in \cite[\S17]{SchneiderBookNFA} (see especially the paragraph following Prop. 17.6 in op. cit.). It follows from the discussion in \cite[\S3]{BuyukbodukLeiFunctionalEquation} that the $p$-adic $L$-function may be thought as a relative $p$-adic $L$-function for $\f$ over $K$, interpolating the algebraic parts of the $L$-values $L(\f(\kappa)_{/K},\Psi,1)$ where $\kappa\in I$ and $\Psi$ runs through the algebraic Hecke characters of $K$ with infinity type $(a,b)$ with $1-\frac{\kappa}{2}\leq a\leq b \leq \frac{\kappa}{2}-1$.

\begin{defn}
We let  $\mathscr{D}_{\f_{/K}}^{\RS}\in \mathscr{H_A}$ denote the $p$-adic distribution obtained by specializing $L_p^{\RS}(\f,\g)$ to the point $\widetilde{\kappa} \in \mathscr{W}(\Gamma_\p)(\mathbb{Q}_p)$ in the weight space for $\g$, corresponding to the weight one specialization $\Theta^{\rm ord}(\mathds{1}_\wp)$.
\end{defn}
The following interpolation property characterizes the distribution $\mathscr{D}_{\f_{/K}}^{\RS}$. 
\begin{theorem}[Loeffler]
\label{THMloefflersinterpolation}
For every $\kappa\in I$, any $j \in \ZZ\cap [1,\kappa-1]$ and all Dirichlet characters $\eta$ of conductor $p^r$ $($we allow $r=0$$)$ we have
\begin{align*}(\psi_\kappa\otimes\eta\chi_\cyc^{j-1})(\mathscr{D}_{\f_{/K}}^{\RS})=(-1)^{j-1}&\times\frac{\mathcal{V}(\f(\kappa)^\circ,\eta,j)^2\,p^{2r(j-1)}\tau(\eta)^2}{\ubeta(\kappa)^{2r}\mathcal{E}(\f(\kappa))\mathcal{E}^*(\f(\kappa))}\\
&\times\frac{i^{\kappa-1}\Gamma(j)^2}{2^{2j+\kappa-1}\pi^{{2j}}}\times\frac{L(\f(\kappa)^\circ_{/K}, \eta^{-1}\circ\NN_{K/\QQ},j)}{\langle \f(\kappa)^\circ,\f(\kappa)^\circ\rangle_N}
\end{align*}
where $\displaystyle{\mathcal{V}(\underline{\bf{f}}(\kappa)^\circ,\eta,j)=\left(1-{p^{j-1}\eta(p)}\big{/}{\ubeta(\kappa)}\right)\left(1-{p^{\kappa-1-j}\eta^{-1}(p)}\big{/}{\ubeta(\kappa)}\right)}$, $\tau(\eta)$ is the Gauss sum normalized to have norm $p^{\frac{r}{2}}$ and finally,
$$\mathcal{E}(\f(\kappa)):=\left(1-\frac{p^{\kappa-2}}{\ubeta(\kappa)^2}\right)\,\,\,\,,\,\,\,\,\mathcal{E}^*(\f(\kappa)):=\left(1-\frac{p^{\kappa-1}}{\ubeta(\kappa)^2}\right)\,.$$
\end{theorem}

\begin{remark}
The $3$-variable Rankin--Selberg $p$-adic $L$-function $L_p^{\RS}(\f,\g\vert_{\mathscr Y})$ is the one given as in~\cite[Theorem 6.1]{LoefflerpadicRankinSelberg} and denoted by $L_p^{\rm geom}(\mathcal{F}_1,\mathcal{F}_2)$ in op.\ cit., with the choices $(\mathcal{F}_1,U_1)=(\f,U)$ and $(\mathcal{F}_2,\widetilde{U}_2)=(\g_{\vert_{\mathscr{Y}}},\mathscr{Y})$. We note that this $p$-adic $L$-function is only meromorphic a priori, but as explained in the Footnote 2 in \cite[\S6]{LoefflerpadicRankinSelberg}, it is analytic since the nebentype character of $\f$ is trivial and that of $\g$ equals $\epsilon_K$ (in particular, their product is non-trivial), by assumption.

We note that the $p$-adic $L$-function constructed in \cite{LoefflerpadicRankinSelberg} interpolates the values of the \emph{imprimitive} Rankin--Selberg $L$-series. Our running assumption is that the tame levels of the two families $\f$ and $\g$ are coprime. Thanks to this condition, the imprimitive Rankin--Selberg $L$-series associated to the crystalline specializations of $\f\otimes\g$ coincide with that of the motivic $L$-functions, c.f. \cite[Remark 2.2]{LoefflerpadicRankinSelberg}.

We briefly review the the construction $L_p^{\RS}(\f,\g\vert_{\mathscr Y})$ for the convenience of the reader. To construct this $p$-adic $L$-function with the desired interpolation properties, Loeffler in \cite[\S4]{LoefflerpadicRankinSelberg} first constructs two families of auxiliary $p$-adic $L$-functions $L_p^{\spadesuit}(\f,\g\vert_{\mathscr Y},\tau)$ and $L_p^{\diamondsuit}(\f,\g\vert_{\mathscr Y},\tau)$ for each choice of algebraic $\tau\in \mathscr{W}(\Gamma_\cyc):={\rm Spm}\, \Zp[[\Gamma_\cyc]]\otimes\Qp$ of non-negative weight (so that over some open subgroup of $\Gamma_\cyc$, we have $\tau=\chi_\cyc^{w(\tau)}$ for some non-negative integer $w(\tau)$), over the two-dimensional rigid analytic subspaces 
\begin{align*}
    \mathscr{W}^{\spadesuit}(\tau)&\,:=\{(\kappa,\widetilde{\lambda},\kappa-1-\tau)\in U\times \mathscr{Y}\times \mathscr{W}(\Gamma_\cyc)\} \\
    \mathscr{W}^{\diamondsuit}(\tau)&\,:=\{(\kappa,\widetilde{\lambda},\lambda+\tau)\in U\times \mathscr{Y}\times \mathscr{W}(\Gamma_\cyc)\}
\end{align*}
of $U\times\mathscr{Y}\times  \mathscr{W}(\Gamma_\cyc)$ (where for $\widetilde{\lambda}\in \mathscr{Y}$, we denote by $\lambda \in \mathscr{W}(\Gamma_\cyc)$ the point that lies below $\widetilde{\lambda}$). The definition of $L_p^{?}(\f,\g\vert_{\mathscr Y},\tau)$ (?=$\spadesuit,\diamondsuit$) follows Hida's original strategy, as enhanced via Urban's overconvergent projector $\Pi^{\rm oc}$ (which interpolates Shimura's holomorphic projector in an appropriate sense), c.f. \cite[\S3.3.4]{urban}. In more precise terms, Loeffler puts
\begin{align*}
    L_p^{\spadesuit}(\f,\g\vert_{\mathscr Y},\tau)&\,:=N^{-\kappa+\widetilde{\lambda}+2\tau}\lambda_{\f}\left[\Pi^{\rm oc}\left(\g_{\vert_{\mathscr Y}}\cdot \theta^{\tau}\left(E^{[p]}_{\kappa-\widetilde{\lambda}-2\tau}\right)\right)\right] \\
    L_p^{\diamondsuit}(\f,\g\vert_{\mathscr Y},\tau)&\,:=N^{\kappa-\widetilde{\lambda}-2\tau-2}\lambda_{\f}\left[\Pi^{\rm oc}\left(\g_{\vert_{\mathscr Y}}\cdot \theta^{\tau}\left(F^{[p]}_{\kappa-\widetilde{\lambda}-2\tau}\right)\right)\right]
\end{align*}
where: 
\begin{itemize}
    \item $\lambda_{\f}: S_{\kappa}^\dagger(\Gamma_0(N),U)\to \cO(U)$ is the $ \cO(U)$-linear functional that factors through the Hecke eigenspace associated to $\f$ and that sends $\f$ to $1$. Note the conjugate Coleman family $\f^c$ (given as in \cite{LoefflerpadicRankinSelberg}, Lemma 3.4) coincides with $\f$, since the tame nebentype of the family $\f$ is trivial by assumption. 
    \item $\theta^\tau$ is the twisting operator on the families of overconvergent modular forms (c.f. Lemma 3.3 in op. cit. and the paragraph that follows its proof).
    \item $E^{[p]}$ and $F^{[p]}$ are the two families of $p$-depleted  overconvergent modular forms, given as in Lemma 3.2 in op. cit. Note in particular that we enlarge our coefficent field $E$ so that it contains $N$th roots of unity. 
\end{itemize}
By construction, for each $?=\spadesuit,\diamondsuit$ and $\tau$ as before, the $p$-adic $L$-function  $L_p^{?}(\f,\g\vert_{\mathscr Y},\tau)$ verifies the required interpolation formula at the classical points of $\mathscr{W}^{?}(\tau)$ with $\kappa>\lambda$. The subspaces $\mathscr{W}^{\spadesuit}(\tau)$ and $\mathscr{W}^{\diamondsuit}(\tau)$ are small enough to ensure that the formal power series $\g_{\vert_{\mathscr Y}}\cdot \theta^{\tau}\left(E^{[p]}_{\kappa-\widetilde{\lambda}-2\tau}\right)$ and $\g_{\vert_{\mathscr Y}}\cdot \theta^{\tau}\left(F^{[p]}_{\kappa-\widetilde{\lambda}-2\tau}\right)$ become nearly-overconvergent over these spaces (so that we can apply Urban's overconvergent projector $\Pi^{\rm oc}$ on them). Moreover, $L_p^{\spadesuit}(\f,\g\vert_{\mathscr Y},\tau)$ and $L_p^{\diamondsuit}(\f,\g\vert_{\mathscr Y},\tau')$ coincide as functions on $\mathscr{W}^{\spadesuit}(\tau)\cap \mathscr{W}^{\diamondsuit}(\tau')$. 

Finally, one would like to patch $L_p^{?}(\f,\g\vert_{\mathscr Y},\tau)$ as $?=\spadesuit, \diamondsuit$ and $\tau$ vary. To achieve this, the first step is to enlarge $\mathscr{Y}$ (which one can, using the fact that the family $\g$ is ordinary) so as to ensure that the union $\bigcup_{t\in \mathbb{N}}\mathscr{W}^{\spadesuit}(t)\cap \mathscr{W}^{\diamondsuit}(\tau)$ is Zariski-dense in $\mathscr{W}^{\diamond}(\tau)$ (c.f. \cite{LoefflerpadicRankinSelberg}, Lemma 5.3). Using this observation and the existence\footnote{Loeffler relies on the earlier work of Loeffler and Zerbes in \cite[\S9.1]{LZ1} at this point. Alternatively, one may of course resort to the results of Andreatta--Iovita and Urban in \cite{urban} and \cite[Appendix B]{AndreattaIovita3Var} if one desires a treatment solely using $p$-adic analytic methods.} of a $3$-variable $p$-adic $L$-function associated to $\f\otimes\g_{\vert_{\mathscr{Y}}}$ that verifies the sought-after interpolation property at a subset of \emph{the crystalline points} of $U\times \mathscr{Y} \times \mathscr{W}(\Gamma_\cyc)$, Loeffler shows in \cite[Corollary~6.2]{LoefflerpadicRankinSelberg} that these families of $p$-adic $L$-functions do patch to an analytic function over the entirety of $U\times \mathscr{Y} \times \mathscr{W}(\Gamma_\cyc)$.
\end{remark}

\subsection{Na\"ive $p$-adic $L$-functions over $K$}

We now consider a na\"ive version of a $p$-adic $L$-function over $K$ by taking the product of two $p$-adic $L$-functions over $\QQ$.  We begin by recalling two-variable $p$-adic $L$-functions over the eigencurve.  This construction is due to Glenn Stevens, but first appeared in the literature in \cite{bellaichepadicL}.

Suppose that  $\underline{\bf{h}}$ is a Coleman family over a sufficiently small affinoid disc ${\rm Sp}(\mathscr{A})$ about a non-$\theta$-critical point $g$ of weight $k_0$ on the eigencurve $($in the sense of Definition 2.12 in \cite{bellaichepadicL}$)$ with $U_p$-eigenvalue $\bbalpha$.   Let $I$ denote the set of classical weights of forms occurring in $\mathscr{A}$.

\begin{theorem}
\label{thmbellaichemain}
There exists a $p$-adic distribution $\mathscr{D}^{\Be}_{\underline{\bf{h}}} \in \mathscr{H_A}$ which verifies the following interpolation property:\ For every $\kappa\in I$, any $j \in \ZZ\cap [1,\kappa-1]$ and all Dirichlet characters $\eta$ of conductor $p^r \geq 1$,
\begin{align*}
(\psi_\kappa\otimes\eta\chi_\cyc^{j-1})(\mathscr{D}^{\Be}_{\underline{\bf{h}}})=(-1)^j\Gamma(j)\mathcal{V}(\underline{\bf{h}}(\kappa)^\circ,\eta,j)\tau(\eta)\frac{p^{(j-1)r}}{\bbalpha(\kappa)^r}\frac{L(\underline{\bf{h}}(\kappa)^\circ,\eta^{-1},j)}{(2\pi i)^{j-1}\Omega_{\underline{\bf{h}}(\kappa)}^{\pm}}C_{\underline{\bf{h}}(\kappa)}^\pm
\end{align*}
where, 
\begin{itemize}
\item $\tau(\eta)$ is the Gauss sum $($normalized to have norm $p^{r/2}$$)$, 
\item  $\displaystyle{\mathcal{V}(\underline{\bf{h}}(\kappa)^\circ,\eta,j)=\left(1-{p^{j-1}\eta(p)}\big{/}{\bbalpha(\kappa)}\right)\left(1-{p^{\kappa-1-j}\eta^{-1}(p)}\big{/}{\bbalpha(\kappa)}\right)}$,
\item $\Omega_{\underline{\bf{h}}(\kappa)}^{+}$ and $\Omega_{\underline{\bf{h}}(\kappa)}^{-}$ are canonical periods in the sense of \cite[\S1.3]{Vatsal_Periods_1999},
\item $C_{\underline{\bf{h}}(\kappa)}^+$ and $C_{\underline{\bf{h}}(\kappa)}^-$ are 
non-zero constants that only depend on $\kappa$ and $C_{\underline{\bf{h}}(k_0)}^+ =C_{\underline{\bf{h}}(k_0)}^-=1$,
\item the sign $\pm$ is determined so as to ensure that $(-1)^{(j-1)}\eta(-1)=\pm1$.
\end{itemize}
Moreover, such a distribution is unique up to multiplication by a non-vanishing locally analytic function of $\kappa$, which accounts for the ambiguity in the choices of the $p$-adic periods $C_{\underline{\bf{h}}(\kappa)}^\pm$. 
\end{theorem}

\begin{proof}
See \cite[Theorem 3 and (4)]{bellaichepadicL}.
\end{proof}

\begin{defn}
\label{defKrelativepadicLfunc}

For the Coleman family $\f$ we have fixed above, mimicking Kobayashi \cite{KobayashiGZ,KobayashiGZ2}, we define $\mathscr{D}^{\naive}_{\f_{/K}}$ as the convolution of the $A$-valued distributions $\mathscr{D}_{\f_{/\QQ}}^{\Be}$ and  $\mathscr{D}_{\f_{/\QQ}^K}^{\Be}$:
$$\mathscr{D}^{\naive}_{\f_{/K}}:=\mathscr{D}_{\f_{/\QQ}}^{\Be}\ast \mathscr{D}_{\f_{/\QQ}^K}^{\Be}\,.$$ 
Here, $\f^K$ is the family obtained by twisting the Coleman family $\f$ by the quadratic character $\epsilon_K$. We call $\mathscr{D}^{\naive}_{\f_{/K}}$ the \emph{na\"{i}ve base change $p$-adic $L$-function}. 
\end{defn}

The na\"ive base change $p$-adic $L$-function is then characterized by the following interpolation property: 

For every $\kappa\in I$, any $j \in \ZZ\cap [1,\kappa-1]$ with $j\equiv k/2$ mod $2$, and all \emph{even} Dirichlet characters $\eta$ of conductor $p^r$ (we allow $r=0$),
\begin{align}
\label{EQNnaivepadicLinterpolation}
\begin{aligned}
(\psi_\kappa\otimes\eta\chi_\cyc^{j-1})(\mathscr{D}^{\naive}_{\f_{/K}})=\Gamma(j)^2\,\mathcal{V}(\f(\kappa)^\circ,\eta,j)^2\,&\tau(\eta)^2\,\frac{p^{2r(j-1)}}{\ubeta(\kappa)^{2r}}\times\frac{C_{\f(\kappa)}^\varepsilon C_{\f(\kappa)^K}^\varepsilon}{\Omega_{\f(\kappa)}^{\varepsilon}\Omega_{\f(\kappa)^K}^{\varepsilon}}\\
&\times\frac{L({\f(\kappa)^\circ_{/K},\eta^{-1}\circ\mathbb{N}_{K/\QQ}},j)}{(2\pi i)^{2(j-1)}}
\end{aligned}
\end{align}
where $\varepsilon\in \{\pm\}$ is the sign of $(-1)^{k/2-1}$.

%%%%%%%%%%%%%%%%%%%%%%%%%%%%%%%%%%%%%%%%%%%%%%%%
%%%%%%%%%%%%%%%%%%%%%%%%%%%%%%%%%%%%%%%%%%%%%%%%
%%%%%%%%%%%%%%%%%%%%%%%%%%%%%%%%%%%%%%%%%%%%%%%%
%%%%%%%%%%%%%%%%%%%%%%%%%%%%%%%%%%%%%%%%%%%%%%%%
%%%%%%%%%%%%%%%%%%%%%%%%%%%%%%%%%%%%%%%%%%%%%%%%

\subsection{A factorization formula}

Let $\eta$ be a Dirichlet character of order $p^r$ (where $r$ is a natural number) and $p$-power conductor. Let $\omega$ denote the Teichm\"uller character and in what follows, denote by $s$ a parameter that varies over $\ZZ_p$.

%--------------------
\begin{defn}
\label{DEF_TwoVarAmiceTransforms}
Fix a Dirichlet character $\eta$ of order $p^r$ and $p$-power conductor as above. 
\item[i)] Given a locally analytic distribution $\mathscr{D}$ on $\Gamma_\cyc\xrightarrow[\chi_\cyc]{\sim}\ZZ_p^\times$, we put 
$$L_p(\mathscr{D},s+\eta):=\langle\chi_\cyc\rangle^{s-1}\omega^{k/2-1}\eta\left(\mathscr{D}\right)$$
for $s\in \ZZ_p$. When $\eta=\mathds{1}$ is the trivial character, we write $s$ in place of $s+\mathds{1}$.  In particular, we shall write $L_p(\mathscr{D},s)$ in place of $L_p(\mathscr{D},s+\mathds{1})$. 
\item[ii)] We define the analytic functions $L_p^{\RS}(\f,\kappa,s+\eta)$ and $L_p^{\naive}(\f_{/K},\kappa,s+\eta)$ on $U\times {\ZZ_p}$ by setting 
\begin{align*}
    L_p^{\RS}(\f_{/K},\kappa,s+\eta)&:=i^{1-k}(-1)^{\frac{k}{2}-1}\frac{\eta(N)^2\langle N\rangle^{2s-\kappa}D_K^{\frac{1}{2}}}{\ubeta(\kappa)\,\lambda_N(\f(\kappa)^\circ)}L_p({\mathscr{D}}_{\f_{/K}}^{\RS},s+\eta)\Big{\vert}_{w=(1+p)^{\kappa-k}-1}\\
    &=-i\frac{\eta(N)^2\langle N\rangle^{2s-\kappa} D_K^{\frac{1}{2}}}{\ubeta(\kappa)\,\lambda_N(\f(\kappa)^\circ)} \times L_p({\mathscr{D}}_{\f_{/K}}^{\RS},s+\eta)\big{\vert}_{w=(1+p)^{\kappa-k}-1}\\
    L_p^{\naive}(\f_{/K},\kappa,s+\eta)&:=L_p({\mathscr{D}}_{\f_{/K}}^{\naive},s+\eta)\big{\vert}_{w=(1+p)^{\kappa-k}-1}\,.
\end{align*}
%Self-remark: I think $i^{1-k}$ (similarly for others) is the only sensible value for $i^{1-\kappa}$, on the connected component of k in the weight space. To make sense of $<N>^{s-\kappa/2}$, write $<N>^{s-\kappa/2}=<\sqrt{N}>^{2s-\kappa}$ for any choice of square root. This definition is independent of the choice. One aside remark is that N is a square in Q_p by the Heegner hypothesis.
\item[iii)] For each choice of $\kappa\in I$, we set
$${L_{p,\ubeta(\kappa)}^{\rm Kob}(\f(\kappa)^\circ_{/K},s+\eta):=\ubeta(\kappa)\,\lambda_N(\f(\kappa)^\circ)\mathcal{E}(\f(\kappa))\, \mathcal{E}^*(\f(\kappa))\,L_p^{\RS}(\f_{/K},\kappa,s+\eta)}$$ % This remains unchanged at the end, given the alteration on $L_p^{\RS}$. We should only see how the latter parts (involving factorization and results towards PR-conjecture) are effected.
In the particular case when $\kappa=k$, we shall write $L_{p,\beta}^{\rm Kob}(f_{/K},s+\eta)$ in place of $L_{p,\ubeta(k)}^{\rm Kob}(\f(k)^\circ_{/K},s+\eta)$.
\end{defn}

See Remark~\ref{remark_GeoVsNek} below for a comparison of $L_{p,\ubeta(\kappa)}^{\rm Kob}(\f(\kappa)^\circ_{/K},s+\eta)$ to Nekov\'a\v{r}'s $p$-adic $L$-function in the $p$-ordinary set-up.
\begin{remark}
\label{REM_KobayashiNekDoesNotInterpolate}
Observe that the $p$-adic multipliers $\mathcal{E}(\f(\kappa))\,\mathcal{E}^*(\f(\kappa))$ do not vary continuously, and the $p$-adic $L$-functions $L_{p,\ubeta(\kappa)}^{\rm Kob}(\f(\kappa)^\circ_{/K},s)$ do not interpolate as $\kappa\in I$ varies. 
\end{remark}

Let us consider the function 
$${R}_{\f/K}(\kappa,\eta,s):=L_p^{\RS}(\f_{/K},s+\eta)\big{/}L_p^{\naive}(\f_{/K},\kappa,s+\eta).$$ 
We shall write ${R}_{\f/K}(\kappa,s)$ in place of ${R}_{\f/K}(\kappa,\mathds{1},s)$.

\begin{lemma}
\label{Lemma_rinterpolatesperiods}
 Let $\eta$ be a character of $\Gamma_\cyc$ with conductor $p^r$ and $p$-power order, and let $\langle\,,\,\rangle_{N,p}$ denote the Petersson inner product at level $\Gamma_0(Np)$. 
 \item[i)] The function $r(\kappa,\eta):={R}_{\f/K}(\kappa,\eta,\frac{k}{2})$ $($in the variable $\kappa$$)$  specializes to 
\begin{align}
\label{EQNpstabilizedpetersson}
 \frac{\eta(N)^2\langle N \rangle^{k-\kappa} D_K^{\frac{1}{2}}}{{8 \pi^2} \langle \f(\kappa),\f(\kappa)\rangle_{N,p}}\times 2^{2-\kappa}(-1)^{\frac{k}{2}-1}\frac{\Omega_{\f(\kappa)}^\varepsilon\,\Omega_{\f(\kappa)^K}^\varepsilon}{C_{\f(\kappa)}^\varepsilon \,C_{\f(\kappa)^K}^\varepsilon}
\end{align}
whenever $k\neq \kappa \in I$ or $\kappa=k$ but $L(f_{/K},\eta^{-1},\frac{k}{2})\neq 0$.
\item[ii)] For every Dirichlet character $\eta$ of $p$-power order and conductor we have
\begin{equation}
    \label{eqn_invariance_under_cyclotomic_translation}
r(\kappa,\eta)= \eta(N)^2 \, r(\kappa,\mathds{1})\,.
\end{equation}
\item[iii)] If $U$ is sufficiently small,  $r(\kappa):=r(\kappa+\mathds{1})$ is analytic on $U$.
\end{lemma}

\begin{proof}
To prove (i), we compare the interpolation formula for the $p$-adic $L$-function $L_p^{\RS}(\f_{/K},\kappa,\sigma)$  at $\sigma=\frac{k}{2}+\eta$ and $\kappa\in I$, which tells us that 
\begin{align}
   \label{eqn_Lemma433_23_09_2021_1}
   \begin{aligned}
     L_p^{\RS}(\f_{/K},\kappa,\frac{k}{2}+\eta)=&\,\Gamma({k}/{2})^2\,\mathcal{V}(\f(\kappa)^\circ,\eta,{k}/{2})^2 \,\tau(\eta)^2\, \frac{p^{2r(\frac{k}{2}-1)}}{\ubeta(\kappa)^{2r}} \\
     &\times \frac{ \eta(N)^2\langle N\rangle^{k-\kappa}D_K^{\frac{1}{2}}2^{1-\kappa}}{\ubeta(\kappa)\lambda_N(\f(\kappa)^{\circ})\mathcal{E}(\f(\kappa))\mathcal{E}^*(\f(\kappa))} \times\frac{L(\f(\kappa)^\circ_{/K},\eta^{-1},\frac{k}{2})}{(2\pi)^{k}\langle  \f(\kappa)^\circ,\f(\kappa)^\circ\rangle_N}
   \end{aligned}
\end{align}
and that of $L_p^{\naive}(\f_{/K},\kappa,s)$, according which we have 
\begin{align}
   \label{eqn_Lemma433_23_09_2021_2}
   \begin{aligned}
   L_p^{\naive}(\f_{/K},\kappa,\frac{k}{2}+\eta)&=\Gamma({k}/{2})^2\,\mathcal{V}(\f(\kappa)^\circ,\eta,{k}/{2})^2\times \tau(\eta)^2\,\frac{p^{2r(\frac{k}{2}-1)}}{\ubeta(\kappa)^{2r}}\\
   &\qquad\times \frac{C_{\f(\kappa)}^\varepsilon C_{\f(\kappa)^K}^\varepsilon}{\Omega_{\f(\kappa)}^{\varepsilon}\Omega_{\f(\kappa)^K}^{\varepsilon}}\times\frac{L(\f(\kappa)^\circ_{/K},\eta^{-1},\frac{k}{2})}{(2\pi i)^{k-2}}.
   \end{aligned}
\end{align}
The formulae \eqref{eqn_Lemma433_23_09_2021_1} and \eqref{eqn_Lemma433_23_09_2021_2} combined with the well-known comparison of Petersson inner products (c.f. \cite{LoefflerpadicRankinSelberg}, Page 619)
\begin{equation}
\label{eqn_relate_Petersson_products}
    \langle \f(\kappa),\f(\kappa)\rangle_{N,p}=\ubeta(\kappa)\lambda_N(\f(\kappa)^{\circ})\mathcal{E}(\f(\kappa))\,\mathcal{E}^*(\f(\kappa))\,\langle \f(\kappa)^\circ,\f(\kappa)^\circ\rangle_N
\end{equation}
 implies the equality in \eqref{EQNpstabilizedpetersson}.
 
We now prove (ii). Let $\eta$ be any Dirichlet character as in the statement of our lemma and let us put 
$$D_\eta(\kappa):=r(\kappa,\eta)-\eta(N)^2 \,r(\kappa).$$ 
Since both $r(\kappa,\eta)$ and $r(\kappa)$ are meromorphic functions of $\kappa$, so is $D_\eta(\kappa)$. Moreover  It follows from part (i) that 
we have 
$$D_\eta(\kappa)=0,\qquad \forall \kappa\in I, \kappa\neq k\,.$$
Since the set $I\setminus\{k\}$ is dense in $U$, this implies that the meromorphic function $D_\eta$ is identically zero and the proof of the identity \eqref{eqn_invariance_under_cyclotomic_translation} follows.
 
 We now prove our final assertion, namely that $r(\kappa)\in \cO(U)$ if the neighborhood $U$ of $k$ is small enough. By part (ii), it suffices to check that $r(\kappa,\eta)\in \cO(U)$ for a Dirichlet character $\eta$ with $L(f_{/K},\eta^{-1},\frac{k}{2})\neq 0$. For such $\eta$, Equation \eqref{EQNpstabilizedpetersson} tells us that 
 $$r(k,\eta)= \frac{ \eta(N)^2 D_K^{\frac{1}{2}}}{8\pi^2 \langle f^\beta,f^\beta\rangle_{N,p}}\times 2^{2-k}(-1)^{\frac{k}{2}-1}\frac{\Omega_{f^\beta}^\varepsilon\,\Omega_{(f^\beta)^K}^\varepsilon}{C_{f^\beta}^\varepsilon \,C_{(f^\beta)^K}^\varepsilon}\,.$$
 In particular, $r(\kappa,\eta)=\frac{a(\kappa)}{b(\kappa)}$ (where $a,b\in \cO(U)$ are coprime) is regular at $\kappa=k$ (i.e., $b(k)\neq 0$). On shrinking $U$ appropriately, we can ensure that $b$ does not vanish on $U$ and for such $U$, we have $r(\kappa,\eta)\in \cO(U)$. This completes the proof of (iii). 
\end{proof}

Until the end, we assume that $U$ is small enough to guarantee the validity of the conclusions of Lemma~\ref{Lemma_rinterpolatesperiods}.

\begin{lemma}
\label{lemma:Omegadependsonlyonweight}
Let $\eta$ be a Dirichlet character of conductor $p^r$ and $p$-power order. For sufficiently small $U$, we have 
$${R}_{\f/K}(\kappa,\eta,s)=\langle N \rangle^{2s-k} \,r(\kappa,\eta)$$
as meromorphic functions on $U\times \ZZ_p$.
\end{lemma}

\begin{proof}
The equation \eqref{EQNnaivepadicLinterpolation} applied taking $j\in [1,\kappa-1]$ with $j\equiv \frac{k}{2} \mod (p-1)$ shows that we have for all $k\neq\kappa\in I$
\begin{align}
    \label{eqn_thmbellaichemain_etais12jequivk}
    \begin{aligned}
L_p^{\naive}(\f_{/K},\kappa,j+\eta)&\,=\Gamma(j)^2\,\mathcal{V}(\f(\kappa)^\circ,\eta,j)^2\times \tau(\eta)^2\,\frac{p^{2r(j-1)}}{\ubeta(\kappa)^{2r}}\\
&\qquad\qquad\times\frac{C_{\f(\kappa)}^\varepsilon C_{\f(\kappa)^K}^\varepsilon}{\Omega_{\f(\kappa)}^{\varepsilon}\Omega_{\f(\kappa)^K}^{\varepsilon}}\times\frac{L(\f(\kappa)^\circ_{/K},\eta^{-1},j)}{(2\pi i)^{2(j-1)}}
\end{aligned}
\end{align}
whereas by Theorem~\ref{THMloefflersinterpolation} combined with Definition~\ref{DEF_TwoVarAmiceTransforms} and \eqref{eqn_relate_Petersson_products}, we have 
\begin{align}
\label{eqn_THMloefflersinterpolationin_etais12jequivk}
\begin{aligned}
  L_p^{\RS}(\f_{/K},\kappa,j+\eta)=
    &\,\Gamma(j)^2\mathcal{V}(\f(\kappa)^\circ,\eta,j)^2\,\tau(\eta)^2 \frac{p^{2r(j-1)}}{\ubeta(\kappa)^{2r}} \\
    &\qquad\times\eta(N)^2\langle N\rangle^{2j-\kappa}\times \frac{2^{1-\kappa} L(\f(\kappa)^\circ_{/K},\eta^{-1},j)}{(2\pi)^{2j}\langle \f(\kappa),\f(\kappa)\rangle_{N,p}}\,.
\end{aligned}
\end{align}
Combining \eqref{eqn_thmbellaichemain_etais12jequivk} and \eqref{eqn_THMloefflersinterpolationin_etais12jequivk} with \eqref{EQNpstabilizedpetersson}, we deduce that
\begin{align}
    \label{eqn_special_values_of_Rboldfkappas}
    \begin{aligned}
    L_p^{\RS}(\f_{/K},\kappa,j+\eta)=\langle N \rangle^{2j-k}  r(\kappa,\eta)L_p^{\naive}(\f_{/K},\kappa,j+\eta)
\end{aligned}
\end{align}
for all $k\neq \kappa\in I$ and $j\in \ZZ\cap [1,\kappa-1]$ with $j\equiv \frac{k}{2} \mod (p-1)$ as above. 

We note that all three functions 
$$L_p^{\RS}(\f_{/K},\kappa,s+\eta), \quad \langle N \rangle^{2s-k} r(\kappa,\eta),\quad L_p^{\naive}(\f_{/K},\kappa,s+\eta)$$
that appear in \eqref{eqn_special_values_of_Rboldfkappas} are continuous functions of $(\kappa,s)$, where the continuity of $r(\kappa,\eta)$ has been verified in the proof of Lemma~\ref{Lemma_rinterpolatesperiods}. Since the set
$$\{(\kappa,j): \kappa\in I,\quad \kappa\neq k \hbox{ and }  j\in \ZZ\cap [1,\kappa-1] \hbox{ with }  j\equiv k/2 \mod (p-1)\}$$
is dense in $U\times \ZZ_p$, it follows from \eqref{eqn_special_values_of_Rboldfkappas} that 
$$L_p^{\RS}(\f_{/K},\kappa,s+\eta)=\langle N \rangle^{2s-k}  r(\kappa,\eta)L_p^{\naive}(\f_{/K},\kappa,s+\eta)$$
as required.
\end{proof}

\begin{corollary} 
\label{cor_important_26_09_2021_1}
Let $\eta$ be a Dirichlet character of conductor $p^r$ and $p$-power order. Then,
\begin{equation}
    \label{eqn_cor_important_26_09_2021_1_statement}
    {R}_{\f/K}(k,\eta,s)=\eta(N)^2\langle N \rangle^{2s-k} (-1)^{\frac{k}{2}-1}|D_K|^{\frac{1}{2}}\frac{\Omega_{f}^\varepsilon\,\Omega_{f^K}^\varepsilon}{2^{k-2}{\cdot 8\pi^2}\left\langle f^\beta,f^\beta\right\rangle_{N,p}}\,.
\end{equation} 
\end{corollary}

\begin{proof}
Suppose that $\eta_0$ is a Dirichlet character of $p$-power order and conductor such that $L(f_{/K},\eta_0^{-1},\frac{k}{2})\neq 0$. Plugging in $\kappa=k$ and $\eta=\eta_0$ in  Lemma~\ref{lemma:Omegadependsonlyonweight}, we see that
\begin{equation}
    \label{eqn_cor_important_26_09_2021_1}
    {R}_{\f/K}(k,\eta_0,s)=\langle N \rangle^{s-\frac{k}{2}} \,r(k,\eta_0)\,.
\end{equation}
Equation \eqref{eqn_cor_important_26_09_2021_1} combined with \eqref{EQNpstabilizedpetersson} shows that 
\begin{equation}
    \label{eqn_cor_important_26_09_2021_2}
    {R}_{\f/K}(k,\eta_0,s)=\eta_0(N)^2\langle N \rangle^{2s-k} (-1)^{\frac{k}{2}-1}D_K^{\frac{1}{2}}\frac{\Omega_{f}^\varepsilon\,\Omega_{f^K}^\varepsilon}{2^{k-2}{\cdot 8\pi^2}\left\langle f^\beta,f^\beta\right\rangle_{N,p}}\,.
\end{equation}
Suppose now that $\eta$ is any Dirichlet character of $p$-power order and conductor. Plugging in 
$\kappa=k$ in  Lemma~\ref{lemma:Omegadependsonlyonweight}, we see that
\begin{equation}
    \label{eqn_cor_important_26_09_2021_3}
    {R}_{\f/K}(k,\eta,s)=\langle N \rangle^{s-\frac{k}{2}} \,r(k,\eta)\,.
\end{equation}
We also infer from \eqref{eqn_invariance_under_cyclotomic_translation} that 
\begin{equation}
    \label{eqn_cor_important_26_09_2021_4}
 r(k,\eta)=\eta\eta_0^{-1}(N)^2\, r(k,\eta_0)\,.
\end{equation}
Equation \eqref{eqn_cor_important_26_09_2021_1} and \eqref{eqn_cor_important_26_09_2021_3} together with \eqref{eqn_cor_important_26_09_2021_4} show that 
\begin{equation}
    \label{eqn_cor_important_26_09_2021_5}
 {R}_{\f/K}(k,\eta,s)=\eta\eta_0^{-1}(N)^2\, {R}_{\f/K}(k,\eta_0,s)\,.
\end{equation}
The identity \eqref{eqn_cor_important_26_09_2021_1_statement} follows on combining \eqref{eqn_cor_important_26_09_2021_2} and \eqref{eqn_cor_important_26_09_2021_5}. %The equality \eqref{eqn_cor_important_26_09_2021_2_statement} is an immediate consequence of \eqref{eqn_cor_important_26_09_2021_1_statement}, on taking $\eta=\mathds{1}$ and $s=k/2$.
\end{proof}

\begin{corollary}\label{cor_Kob_vs_Stevens_at_f}
$${\displaystyle L_{p,\beta}^{\rm Kob}(f_{/K},s)=\langle N \rangle^{2s-k}(-1)^{\frac{k}{2}-1}|D_K|^{\frac{1}{2}}\frac{\Omega_{f}^\varepsilon\,\Omega_{f^K}^\varepsilon}{2^{k-2}{\cdot 8\pi^2}\left\langle f,f\right\rangle_{N}}\,L_p^{\Be}(f^\beta_{/\QQ},s)L_p^{\Be}(f^{\beta,K}_{/\QQ},s)\,.}$$
\end{corollary}
\begin{proof}
This is an immediate consequence of Corollary~\ref{cor_important_26_09_2021_1}, on recalling that our choices enforce the requirement that 
$C_{f^\beta}^\pm=C_{f^{\beta,K}}^\pm=1$.  
\end{proof}

\begin{remark}
\label{remark_GeoVsNek}
Only in this remark, $\underline{\bf{h}}$ denotes a primitive Hida family of tame level $N$ and $U_p$-eigenvalue $\bbalpha$. We let $h$ denote its specialization to weight $2r$; suppose $h$ is old at $p$ and let us write $\alpha$ for the $U_p$-eigenvalue on $h$. In this situation, Nekov\'a\v{r} in \cite[I.5.10]{NekovarGZ} constructed a two-variable $p$-adic $L$-function $L_p(h\otimes K,\,\cdot\,)$ associated to $h$. We let 
$$L_p^{\rm Nek}(h_{/K},s+\eta):=L_p(h\otimes K, \eta\langle\chi_\cyc\rangle^{s}\circ \mathbb{N}_{K/\QQ})$$ 
denote its restriction to cyclotomic characters, where as above, $\eta$ is a Dirichlet character of $p$-power order and conductor. 

In this particular case, the distribution $\mathscr{D}^{\RS}_{\underline{\bf{h}}}$ was constructed by Hida and it enjoys an interpolation property that is identical to one recorded in Theorem~\ref{THMloefflersinterpolation}. One may specialize $L_p^{\RS}(\underline{\bf{h}}_{/K},\kappa,s+\eta)$ to the $p$-stabilized form $h$ and obtain a $p$-adic $L$-function $L_{p,\alpha}^{\rm Kob}({{h}^\circ}_{/K},s+\eta)$ as above. One may compare\footnote{To do so: 
\begin{itemize}
\item In Nekov\'a\v{r}'s interpolation formula in \cite[Page 611]{NekovarGZ}, one takes $\mathscr{C}=\mathds{1}$ and $\mathscr{W}:=\eta\circ\mathbb{N}_{K/\QQ}$. 
    \item One expresses the root number denoted by $\tau(\eta\circ \mathbb{N})$ in \cite{NekovarGZ} (which appears in the interpolation formula on Page 611 of op. cit.) in terms of the Gauss sum $\tau(\eta)$ as follows (c.f. \cite{NekovarGZ}, \S5.13):
$\tau(\eta\circ\mathbb{N}_{K/\QQ})= \eta(D_K)\tau(\eta)^2/p^r$\,. 
\end{itemize}
 Then the expression $\tau(\mathscr{W})\,\overline{\mathscr{W}}((\sqrt{-D_K}))$ in op.cit. that appears in the interpolation formula on Page 611 equals 
$$\tau(\eta\circ\mathbb{N}_{K/\QQ})\eta^{-1}(D_K)=\tau(\eta)^2/p^r\,.$$
} the interpolation formulae for the respective distributions giving rise to $L_p^{\rm Nek}(h_{/K},s+\eta)$ and $L_{p,\alpha}^{\rm Kob}({{h}^\circ}_{/K},s+\eta)$ to deduce that
$$L_{p,\alpha}^{\rm Kob}({{h}^\circ}_{/K},s+r+\eta)=L_p^{\rm Nek}({{h}}_{/K},s+\eta)\,.$$
\end{remark}

%%%%%%%%%%%%%%%%%%%%%%%%%%%%%%%%%%%%%%%%%%%%%%%%%%%%%%%%%%%%%%%%%%%%%%%%%%%%%%%%%%%%%%%%%%%%%%%%%%%%%%%%%%%%%%%%%%%%%%%%%%%%%%%%%%%%%%%%%%%%%%%%%%%%%%%%%%%%%%%%%%%%%%%%%%%%%%%%%%%%%%%%%%%%%%%%%%%%%%%%%%%%%%%%%%

\section{Proofs of Theorem~\ref{thm:padicGZoverKIntro} and Corollary~\ref{cor:padicGZoverQIntro}}
\label{sec:proofofCriticalGZ}
We shall assume until the end of this article that $K\neq \QQ(i), \QQ(\sqrt{-3})$. Notice in particular that $u_K=1$.
\subsection{$p$-adic Gross--Zagier formula for non-ordinary eigenforms at non-critical slope}
Suppose $g=\sum_{n=1}a_n(g)q^n\in S_{2r}(\Gamma_0(N))$ is a normalized eigenform. We let $a,b\in \overline{\QQ}$ denote the roots of its Hecke polynomial $X^2-a_p(g)X+p^{2r-1}$ at $p$. Suppose that $v_p(\iota_p(a_p(g)))>0$ and assume that
$$0<h:=v_p(\iota_p(b))< v_p(\iota_p(a))$$
so that we have $2h<2r-1$. Let $g^b \in S_{2r}(\Gamma_0(Np))$ denote the $p$-stabilization corresponding to the Hecke root $b$ and let $\g$ be a Coleman family which admits $g^b$ as its specialization in weight $2r$. Theorem~\ref{THMloefflersinterpolation} applies and equips us with a two-variable $p$-adic $L$-function $L_p^{\RS}(\g_{/K},\kappa,s)$. Let us set 
$$L_{p,b}^{\rm Kob}(g_{/K},s):=b\,\lambda_N(g)\mathcal{E}(g)\, \mathcal{E}^*(g)\,L_p^{\RS}(\g_{/K},2r,s)$$
as in Definition~\ref{DEF_TwoVarAmiceTransforms}. The following $p$-adic Gross--Zagier formula is Kobayashi's work~\cite{kobayashi_higherweight_nonord_GZ} in progress.

\begin{theorem}[Kobayashi]
\label{thm:noncriticalGZhigherweight}
$$\frac{d}{ds}L_{p,b}^{\rm Kob}(g_{/K},s)\big{\vert}_{s=r}=\left(1-\frac{p^{r-1}}{b}\right)^{4}\frac{h_{b,K}^{\rm Nek}(z_g,z_g)}{(4{|D_K|})^{r-1}}\,.$$
\end{theorem}

Recall that we have to assume $K\neq \QQ(i),\QQ(\sqrt{-3})$ since we rely on Kobayashi's results here, so that $u_K=1$.

\begin{corollary}
\label{cor:padicGZoverKThatInterpolates} Suppose that $\kappa\in I$ as in \S\ref{subsec:specializeheights} with $\kappa\geq 2k$.
\item[i)]$\psi_\kappa\circ \mathfrak{h}_{\f,K}\left(\mathscr{Z}_{\f},\mathscr{Z}_{\f}\right)=\dfrac{\left(1-\frac{p^{\frac{\kappa}{2}-1}}{b}\right)^{4}}{(4|D_K|)^{\frac{\kappa}{2}-1}}\dfrac{h_{\ubeta(\kappa),K}^{\rm Nek}(z_{\f(\kappa)^\circ},z_{\f(\kappa)^\circ})}{\ubeta(\kappa)\lambda_N(\f(\kappa)^\circ)\mathcal{E}(\f(\kappa))\mathcal{E}^*(\f(\kappa))}$.
\item[ii)]$\dfrac{d}{ds}{L_{p}^{\RS}(\f_{/K},\kappa,s)}\big{\vert}_{s=\frac{\kappa}{2}}= \psi_\kappa\circ \mathfrak{h}_{\f,K}\left(\mathscr{Z}_{\f},\mathscr{Z}_{\f}\right).$
\end{corollary}
\begin{proof}
\item[i)] To lighten our notation, let us put $\mathfrak{h}_{\f(\kappa),\ubeta(\kappa),K}:=\mathfrak{h}$. We then have
\begin{align*}
\psi_\kappa\circ \mathfrak{h}_{\f,K}\left(\mathscr{Z}_{\f},\mathscr{Z}_{\f}\right)&=\mathfrak{h}\left(\psi_\kappa(\mathscr{Z}_{\f})\,,\,\psi_\kappa(\mathscr{Z}_{\f})\right)\\
&=\dfrac{\left(1-\frac{p^{\frac{\kappa}{2}-1}}{b}\right)^{2}}{(2\sqrt{-D_K})^{\frac{\kappa}{2}-1}}\mathfrak{h}\left(\psi_\kappa(\mathscr{Z}_{\f}),\,\frac{W_{Np}\circ({\rm Pr}_{\ubeta(\kappa)})^*({z_{\f(\kappa)^\circ}})}{\ubeta(\kappa)\lambda_N(\f(\kappa)^\circ)\mathcal{E}(\f(\kappa))\mathcal{E}^*(\f(\kappa))} \right)\\
&=\dfrac{\left(1-\frac{p^{\frac{\kappa}{2}-1}}{b}\right)^{2}}{(2\sqrt{-D_K})^{\frac{\kappa}{2}-1}}\dfrac{h_{\ubeta(\kappa),K}^{\rm Nek}\left(({\rm Pr}_{\ubeta(\kappa)})_*\circ\psi_\kappa(\mathscr{Z}_{\f}),\, W_N(z_{\f(\kappa)^\circ}) \right)}{\ubeta(\kappa)\lambda_N(\f(\kappa)^\circ)^2\mathcal{E}(\f(\kappa))\mathcal{E}^*(\f(\kappa))}\\
&=\dfrac{\left(1-\frac{p^{\frac{\kappa}{2}-1}}{b}\right)^{4}}{(4|D_K|)^{\frac{\kappa}{2}-1}}\dfrac{h_{\ubeta(\kappa),K}^{\rm Nek}\left(W_N(z_{\f(\kappa)^\circ}),\, W_N(z_{\f(\kappa)^\circ}) \right)}{\ubeta(\kappa)\lambda_N(\f(\kappa)^\circ)^3\mathcal{E}(\f(\kappa))\mathcal{E}^*(\f(\kappa))}\\
&=\dfrac{\left(1-\frac{p^{\frac{\kappa}{2}-1}}{b}\right)^{4}}{(4|D_K|)^{\frac{\kappa}{2}-1}}\dfrac{h_{\ubeta(\kappa),K}^{\rm Nek}\left(z_{\f(\kappa)^\circ},\, z_{\f(\kappa)^\circ} \right)}{\ubeta(\kappa)\lambda_N(\f(\kappa)^\circ)\mathcal{E}(\f(\kappa))\mathcal{E}^*(\f(\kappa))}
\end{align*}
where the first equality follows from the upper half of the diagram \eqref{eqn:comaprisonofheights_2}, the second from \eqref{eqn_bigheegnermain_1}, the third from the lower half of the diagram \eqref{eqn:comaprisonofheights_2} and the fact that $\Lambda_N(\f(\kappa)^\circ)^{-1}W_N$ is a self-adjoint involution (see also Remark~\ref{rem_compare_V_g_V_g_star_Atkin} for the usage of the involution $\Lambda_N(\f(\kappa)^\circ)^{-1}W_N:V_{\f(\kappa)}\stackrel{\sim}{\lra}V_{\f(\kappa)}$), the forth from \eqref{eqn_bigheegnermain_2} and last by the fact that $\Lambda_N(\f(\kappa)^\circ)^{-1}W_N$ is a self-adjoint involution. This completes the proof of the first part.
\item[ii)] The asserted equality is a restatement of Theorem~\ref{thm:noncriticalGZhigherweight}(ii), taking (i) into account and recalling from Definition~\ref{DEF_TwoVarAmiceTransforms} the relation between $L_{p}^{\RS}(\f_{/K},\kappa,s)$ and $L_{p,\ubeta(\kappa)}^{\rm Kob}(\f(\kappa)^\circ_{/K},s)$.
\end{proof}

\begin{remark}
Note that the assumption that $\kappa\geq 2k$ guarantees that we have $2v_p(\ubeta(\kappa))<\kappa-1$, as required to apply Kobayashi's Theorem~\ref{thm:noncriticalGZhigherweight}. 

The reason why we record this alteration of Theorem~\ref{thm:noncriticalGZhigherweight} here is because both sides of the asserted equality in Corollary~\ref{cor:padicGZoverKThatInterpolates}(ii) interpolate well as $\kappa$ varies (unlike its predecessor Theorem~\ref{thm:noncriticalGZhigherweight}). See {Remark~\ref{REM_KobayashiNekDoesNotInterpolate}}.  
\end{remark}

\subsection{$\mathscr{A}$-adic Gross--Zagier formula}
Recall the Coleman family $\f$ over the affinoid algebra $\mathscr{A}=\mathscr{A}(U)$ from \S\ref{subsubsec:colemanfamilyintro}. Recall also the $\mathscr{A}$-valued cyclotomic height pairing $\mathfrak{h}_{\f,K}$ we have introduced in \S\ref{subsec:Aadicheights} and the universal Heegner point $\mathscr{Z}_{\f}\in (V_{\f_{/K}},\mathbb D_{\ubeta})$ given as in Theorem~\ref{thm:bigheegnermain}. Recall finally also the two-variable $p$-adic $L$-function $L_p^{\RS}(\f_{/K},\kappa,s)$ from \S\ref{sec:padicLfunctions}.

\begin{defn}
\label{DEF_AmiceTransformOfUniversalHeight}
Let us write $H_{\f,K}$ for the Amice transform of the height pairing $\mathfrak{h}_{\f,K}$. In explicit terms, 
$$H_{\f,K}(x,y):=\mathfrak{h}_{\f,K}(x,y)\big{\vert}_{w=(1+p)^{\kappa-k}-1}\,.$$
\end{defn}
\begin{theorem}[$\mathscr{A}$-adic Gross--Zagier formula]
 With the notation as above, the following identity is valid in $\mathscr{A}:$
\label{thm:twovarGZ} 
$$\frac{\partial}{\partial s}L_p^{\RS}\left(\f_{/K},\kappa,s+\frac{\kappa-k}{2}\right)\Big{\vert}_{s=\frac{k}{2}}=H_{\f,K}(\mathscr{Z}_{\f},\mathscr{Z}_{\f})\,.$$
\end{theorem}
\begin{proof}
Consider the difference 
$$\mathfrak{D}(\kappa):=\frac{\partial}{\partial s}L_p^{\RS}\left(\f_{/K},\kappa,s+\frac{\kappa-k}{2}\right)\Big{\vert}_{s=\frac{k}{2}}-{H}_{\f,K}(\mathscr{Z}_{\f},\mathscr{Z}_{\f})\,.$$
It follows from the interpolative properties of $L_p^{\RS}(\f_{/K},\kappa,s)$, the $\mathscr{A}$-adic height pairing $\mathfrak{h}_{\f,K}$ outlined in \eqref{eqn:comaprisonofheights} and that of the universal Heegner cycle (Theorem~\ref{thm:bigheegnermain}) together with Corollary~\ref{cor:padicGZoverKThatInterpolates}(ii) that
 $$\mathfrak{D}(\kappa)=0,\,\, \forall\,\, \kappa\in I\cap \ZZ_{\geq 2k}\,.$$
By the density of $ I\cap \ZZ_{\geq 2k}$ in the affinoid $U$, we conclude that $\mathfrak{D}$ is identically zero, as required.
\end{proof}
\subsection{Proof of Theorem~\ref{thm:padicGZoverKIntro}}\label{subsec_proof_of_Cor_thm_padicGZoverKIntro} On specializing the statement of Theorem~\ref{thm:twovarGZ} ($\mathscr{A}$-adic Gross--Zagier formula) to $\kappa=k$ and relying once again on the interpolative properties of the $\mathscr{A}$-adic height pairing $\mathfrak{h}_{\f,K}$ and that of the universal Heegner cycle $\mathscr{Z}_{\f}$, the proof of Theorem~\ref{thm:padicGZoverKIntro} follows at once. \qed

\subsection{Proof of Corollary~\ref{cor:padicGZoverQIntro}}
\label{subsec_proof_of_Cor_cor_padicGZoverQIntro}
Recall that we are assuming that the weight $k=2$. Recall also that $A_f/\QQ$ stands for the abelian variety of $\GL_2$-type that the Eichler-Shimura congruences associate to $f$ and that we assume that $L(f/\QQ,s)$ has a simple zero at $s=1$.

It follows from the classical (complex) Gross--Zagier formula and Theorem~\ref{thm:padicGZoverKIntro} that
\begin{equation}
\label{eqn:compareGZformulaecomplexvspadic} 
\frac{d}{ds}L_{p,\beta}^{\rm Kob}(f_{/K},s)\big{|}_{s=1}=\left(1-\frac{1}{\beta}\right)^4\frac{L^\prime(A_f/K,1)\,|D_K|^{\frac{1}{2}}}{\langle P_{f}, P_f\rangle_{\infty,K}\,8\pi^2\,\langle f,f\rangle_{N}}h_{\beta,K}^{\rm Nek}(P_f,P_f)
\end{equation}
where we recall that $P_f \in A_f(K)$ is the Heegner point and $\langle\,,\,\rangle_{\infty,K}$ is the N\'eron-Tate height pairing over $K$. Since we know in our set-up that ${\rm Tr}_{K/\QQ}P_f$ is non-torsion, it is a non-zero multiple of $P$ within the one-dimensional $\QQ$-vector space $A_f(\QQ)\otimes\QQ$. We may therefore replace in (\ref{eqn:compareGZformulaecomplexvspadic}) the height pairings of $P_f$ over $K$ with those of $P$ over $\QQ$ to deduce that 
\begin{equation}
\label{eqn:compareGZformulaecomplexvspadic2}
\frac{d}{ds}L_{p,\beta}^{\rm Kob}(f_{/K},s)\big{|}_{s=1}=\left(1-\frac{1}{\beta}\right)^4\frac{L^\prime(A_f/K,1)|D_K|^{\frac{1}{2}}}{\langle P,P\rangle_{\infty}\,8\pi^2\,\langle f,f\rangle_{N}}h_{\beta,\QQ}^{\rm Nek}(P,P).
\end{equation}
On the other hand, we have 
\begin{equation}
\label{eqn:LfunctionoverKvsQ} 
L^\prime(A_f/K,1)=L^\prime(A_f/\QQ,1)L(A_f^K/\QQ,1)
\end{equation} 
and by Corollary~\ref{cor_Kob_vs_Stevens_at_f} together with the interpolation property of $L_{p,\beta}(f^K_{/\QQ},s)$ (c.f. Theorem~\ref{thmbellaichemain}),
\begin{align}
\label{eqn:padicLoverKvsQ} \begin{aligned}
\frac{d}{ds}L_{p,\beta}^{\rm Kob}(f_{/K},s)\big{|}_{s=1}={\frac{\Omega_f^+\Omega^+_{f^K}|D_K|^{\frac{1}{2}}}{8\pi^{2}\langle f,f\rangle_N}}&\times\frac{d}{ds}L_{p,\beta}^{\Be}(f_{/\QQ},s)\big{|}_{s=1}\\
&\times\left(1-\frac{1}{\beta}\right)^2\frac{L(A_f^K/\QQ,1)}{\Omega^+_{f^K}}\,.
\end{aligned}
\end{align}
Plugging the identities \eqref{eqn:LfunctionoverKvsQ} and \eqref{eqn:padicLoverKvsQ} in \eqref{eqn:compareGZformulaecomplexvspadic2}, the desired equality follows. \qed

%%%%%%%%%%%%%%%%%%%%%%%%%%%%%%%%%%%%%%%%%%%%%%%%%%%%%%%%%%%%%%%%%%%%%%%%%%%%%%%%%%%%%%%%%%%%%%%%%%%%%%%%%%%%%%%%%%%%%%%%%%%%%%%%%%%%%%%%%%%%%%%%%%%%%%%%%%%%%%%%%%%%%%%%%%%%%%%%%%%%%%%%%%%%%%%%%%%%%%%%%%%%%%%%%%%%%%%%%%%%%%%%%%%%%%%%%%%%%%%%%%%%%%%%%%%%%%%%%%%%%%%%%%%%%%%%%%%%%%%%%%%%%%%%%%%%%%%%%%%%%%%%%%%%%%%%%%%%%%

\section{Applications}
\label{sec:applications}
We shall illustrate various applications of the $p$-adic Gross--Zagier formula at critical slope (Theorem~\ref{thm:padicGZoverKIntro} and Corollary~\ref{cor:padicGZoverQIntro}). These were already recorded in \S\ref{sec:Intro} as Theorems~\ref{thm:PRconjectureIntro}, \ref{thm:padicheightnontrivialIntro} and \ref{thm:BSDformulainrankoneIntro}. Before we give proofs of these results, we set some notation and record a number of preliminary results.

\subsection{Perrin-Riou's big logarithm map and $p$-adic $L$-functions}
\label{subsec_PRbigLog_padicLBK}
Until the end of this article, we assume that $f=f_A\in S_2(\Gamma_0(N))$ is a newform that comes attached to an elliptic curve $A/\QQ$. 

We assume that $A$ is non-CM. According to Serre~\cite{serre68}, this in turn means that the local representation $\rho_f\,{\vert_{G_{\QQ_p}}}$ is indecomposable and by a result of Breuil--Emerton~\cite{BreuilEmerton2010}, this in turn translates to the requirement that $f$ does not admit a $\theta$-critical $p$-stabilization. 

We assume also that the residual representation $\overline{\rho}_f=\overline{\rho}_A$ is absolutely irreducible. This final assumption amounts to the requirement that $E$ does not admit a rational $p$-isogenty, which is guaranteed thanks to Mazur's work~\cite{mazur78} whenever $p\not\in \{3,5,7,13,37\}$. Our assumption that $L(f/\QQ,s)$ has a simple zero at $s=1$ is also in effect until the end. 

Set $T:=(\varprojlim A(\overline{\QQ})[p^n])$ and $V=T\otimes_{\Zp}\Qp$. Let $\mathscr{A}_{/\ZZ}$ be a minimal Weierstrass model of the elliptic curve $A$. Let $\omega_{\mathscr A}$ denote a N\'eron differential that is normalized so that $\Omega^+_{A}:=\int_{A(\mathbb{R})}\omega_{\mathscr{A}}>0$. Let $\pi: X_0(N)\to E$ be a modular parametrization of $E$ and let $c_\pi$ denote the Manin constant, so that we have $\pi^*\omega_{\mathscr A}=c_\pi 2\pi i  f(z)dz=c_\pi\omega_f$. In our setting, it follows from~\cite[Corollary 4.1]{mazur78}, which states that if $p\mid c_\pi$ then $p^2\mid 4N$, that $c_\pi$ is a $p$-adic unit.

The Heegner point $P_f\in E(K)\otimes \QQ$ is given as 
$$P_f:=c_{\pi}^{-1}\,{\rm tr}_{H_K/K}\,\pi(z)$$
where $z\in X_0(N)(H_K)$ is any choice of a Heegner point of conductor $1$. Our requirement that $\ord_{s=1}L(E,s)=1=\ord_{s=1}L(E/K,s)$ together with the fact $c_\pi\in \ZZ_p^\times$ shows that
$$P_f\in  E(\QQ)\otimes \ZZ_p\xrightarrow{\sim} H^1_{\rm f}(\QQ,T).$$
Let $P\in E(\QQ)$ be any lift of a generator of the cyclic group $E(\QQ)/E(\QQ)_{\rm tor}$ and let $d\in \QQ^\times$ denote the unique rational number with 
\begin{equation}
    \label{eqn_P_vs_Pf_final_section}
    P_f=P\otimes d\quad \in\, E(\QQ)\otimes\QQ\,.
\end{equation}
Note that $d\in \Zp^\times$, but we will not rely on this observation in what follows.

Let $\omega_{\rm cris}\in \DD_{\rm cris}(V)$ denote the element that corresponds to $\omega_{\mathscr A}$ under the comparison isomorphism. We choose a $\varphi$-eigenbasis $\{\omega_\alpha,\omega_\beta\}$ of $\DD_{\rm cris}(V)$ (where $\omega_\lambda$ is a $\varphi$-eigenvector with the eigenvalue $\lambda^{-1}$) by the requirement that
$$\omega_\alpha+\omega_\beta=\omega_{\rm cris}\,.$$
Note that such an eigenbasis exists since we are assuming that $A$ is non-CM.

We denote by 
$$\mathfrak{Log}_V: H^1(\QQ_p,V\otimes\LL^{\iota})\otimes \mathscr{H} \lra \DD_{\textup{cris}}(V)\otimes\mathscr{H} $$
Perrin-Riou's big dual exponential map, and by
$$\mathfrak{L}_{\rm PR}:=\mathfrak{Log}_V\circ\res_p \left(\mathbb{BK}_1\right) \in \DD_{\textup{cris}}(V)\otimes  \mathscr{H}$$
Perrin-Riou's vector valued $p$-adic $L$-function where $\mathbb{BK}_1 \in  H^1(\QQ,V\otimes\LL^{\iota})$ is the normalized Beilinson--Kato element as in~\cite[\S13.9]{kato04}. Note that the Beilinson--Kato element depends on the choice of a Shimura period $\Omega_f^\pm$,  we fix this choice on requiring that $\Omega_f^\pm=\Omega_A^\pm$.

We define $\mathfrak{L}_{\rm PR}^{(\alpha)},\mathfrak{L}_{\rm PR}^{(\beta)}\in \mathscr{H}_E$ as the coordinates of $\mathfrak{L}_{\rm PR}$ with respect to the basis $\{\omega_\alpha,\omega_\beta\}$, so that we have
$$\mathfrak{L}_{\rm PR}=\mathfrak{L}_{\rm PR}^{(\alpha)}\,\,\omega_\alpha+\mathfrak{L}_{\rm PR}^{(\beta)}\,\,\omega_\beta\,.$$

\begin{defn}
Set $\mathscr{D}_{f^\beta_{/\QQ}}:=\psi_2(\mathscr{D}_{\f_{/\QQ}}^{\Be}) \in \mathscr{D}(\Gamma)$. Associated to the other ($p$-ordinary) stabilization $f^\alpha$ of $f$, we also have the Mazur--Swinnerton-Dyer measure $\mathscr{D}_{f^\alpha_{/\QQ}}\in \mathscr{D}_0(\Gamma)$. 

For $\lambda=\alpha,\beta$, we set $L_{p,\lambda}(f,s):=L_p(\mathscr{D}_{f^\lambda_{/\QQ}},s)$.
\end{defn}

We remark that the measure $\mathscr{D}_{f^\alpha_{/\QQ}}$ is characterized by an interpolation formula identical to that for the distribution $\mathscr{D}_{f^\beta_{/\QQ}}$ (c.f. Theorem~\ref{thmbellaichemain}), exchanging every $\alpha$ in the formula with $\beta$ and vice versa (but we recall that the interpolation property in Theorem~\ref{thmbellaichemain} does \emph{not} characterize $\mathscr{D}_{f^\beta_{/\QQ}}$ itself).

The following result is due to Kato (when $\lambda=\alpha$) and to Benois and the first named author~\cite[Theorem~A]{BenoisBuyukbodukCKPRVolume} (when $\lambda=\beta$, see also the related work by Hansen~\cite{hansenBSP}, Ochiai~\cite{OchiaiIwasawa2018} and Wang~\cite{Wang2021} in this scenario).
\begin{theorem}
\label{thm:katohansen}
For each $\lambda\in \{\alpha,\beta\}$ we have $\mathfrak{L}_{\rm PR}^{(\lambda)}= \mathscr{D}_{f^\lambda_{/\QQ}}\,.$
\end{theorem}

\subsubsection{Proof of Theorem~\ref{thm:padicheightnontrivialIntro} (non-triviality of $p$-adic heights)}
\label{subsubsec_proof_of_thm_padicheightnontrivialIntro}
Suppose on the contrary that both $h_{\alpha,\QQ}^{\rm Nek}$ and $h_{\beta,\QQ}^{\rm Nek}$ were trivial. It follows from Corollary~\ref{cor:padicGZoverQIntro}  and Perrin-Riou's $p$-adic Gross--Zagier formula for the slope-zero $p$-adic $L$-function $L_{p}(\mathscr{D}_{f^\alpha{/\QQ}},s)$ that 
$$\mathds{1}\left(\mathfrak{L}_{\rm PR}^\prime\right)=0\,.$$ 
Using \cite[Proposition 2.2.2]{PerrinRiou93Rubinsformula}, we conclude that $\log_V(\rm{BK}_1)=0$, or equivalently, that $\res_p(\rm{BK}_1)=0$. Since ${\rm ord}_{s=1}L(f/\QQ,s)=1$, the theorem of Kolyvagin  shows that the compositum of the arrows
$$A(\QQ)\,{\otimes}_{\ZZ}\Qp\stackrel{\sim}{\lra} H^1_{\rm f}(\QQ,V)\stackrel{\res_p}{\lra} H_{\rm f}(\QQ_p,V)=A(\QQ_p)\otimes_{\ZZ} \Qp$$
is injective. It follows that ${\rm BK}_1\in H^1_{\rm f}(\QQ,V)$ is a torsion class, contradicting \cite[Theorem 1.2]{kbbIwasawa2017}.
\qed

\subsection{Birch and Swinnerton-Dyer formula for analytic rank one (Proof of Theorem~\ref{thm:BSDformulainrankoneIntro})}
\label{sec:BSD}
We retain the set-up and hypotheses in \S\ref{subsec_PRbigLog_padicLBK}. In particular, we continue to assume that ${\ord}_{s=1}L(A_{/\QQ},s)=1$.

\begin{proof}[Proof of Theorem~\ref{thm:BSDformulainrankoneIntro}] Since the Iwasawa main conjecture for $A$ holds true under our running assumptions (c.f. Remark~\ref{intro_remark_IMC}), Perrin-Riou's leading term formulae for her module of $p$-adic $L$-functions\footnote{Note that the element $\mathfrak{L}_{\rm PR}$ we have introduced above is a generator of this module since we assume the truth of main conjectures.} in \cite[\S 3]{PerrinRiou93Rubinsformula} together with Theorem~\ref{thm:katohansen} and Theorem~\ref{thm:padicheightnontrivialIntro} show that the $p$-adic Birch and Swinnerton-Dyer conjecture (which corresponds to the statement ${\rm BSD}_{\DD_\lambda}(V)$ in \cite[Proposition 3.4.6]{PerrinRiou93Rubinsformula}) for $A$ is true up to $p$-adic units:
\begin{align}\label{EQNBSDfrakPadic}
{\rm ord}_{p}\left(\left(1-{1}/{\lambda}\right)^{-2} \frac{ L_{p,\lambda}^\prime(f/\QQ,1)}{{\rm Reg}_{p}(A/\QQ)}\right)={\rm length}_{\Zp}&\left(\hbox{III}(A/\QQ)[p^\infty]\right) \\
\notag&\,\,\,\,\,\,\,\,\,\,\,\,\,\,\,\,\,\,+{\rm ord}_{p}{\rm Tam}(A/\QQ)\,.
\end{align}
Here, ${\rm Tam}(A/\QQ):=\prod_{\ell\mid N}c_\ell(A/\QQ)$ and $c_\ell(A/\QQ)$ is the Tamagawa factor at $\ell$, and
\begin{equation}\label{EQN_sigmapartofregulator}
{\rm Reg}_{p}(A/\QQ):=h_{\lambda,\QQ}^{\rm Nek}(P,P)=\frac{h_{\lambda,\QQ}^{\rm Nek}(P_{f},P_{f})}{d^2}\,,
\end{equation} 
where we recall that $P\in E(\QQ)$ is a lift of a generator of $E(\QQ)/E(\QQ)_{\rm tor}$ and $d\in \QQ^\times$ is as in \eqref{eqn_P_vs_Pf_final_section}.  Notice that the terms concerning the torsion group $A(\QQ)[p^\infty]$ are omitted from \eqref{EQNBSDfrakPadic}, as they are both trivial since we assume that $\overline{\rho}_f=A(\overline{\QQ})[p]$ is absolutely irreducible by assumption.

Applying either Corollary~\ref{cor:padicGZoverQIntro} (if the $p$-adic valuation of $\lambda$ is $1$) or Perrin-Riou's $p$-adic Gross--Zagier formula at slope-zero (if $\lambda$ is a $p$-adic unit), we see that
\begin{equation}\label{EQNGZquotientsatsigma}
-\frac{L^\prime(f/\QQ,1)}{{\rm Reg}_{\infty}(A/\QQ)\,{\Omega_{f}^+}}= \left(1-{1}/{\lambda}\right)^{-2} \frac{L_{p,\lambda}^\prime(f,1)} {{\rm Reg}_{p}(A/\QQ)} 
\end{equation}
 where ${\rm Reg}_{\infty}(A/\QQ):={\langle P_{f},P_{f}\rangle_\infty}/d^2$. Combining \eqref{EQNBSDfrakPadic} and \eqref{EQNGZquotientsatsigma}, we infer that 
\begin{align*}
{\rm ord}_{p}\left(-\frac{L^\prime(f/\QQ,1)}{{\rm Reg}_{\infty}(A/\QQ) \,{\Omega_{f}^+}}\right)={\rm length}_{\Zp}&\left(\hbox{III}(A/\QQ)[p^\infty]\right) +{\rm ord}_{p}{\rm Tam}(A/\QQ)\,.
\end{align*}
Since $L(f,s)=L(E/\QQ,s)$ and $\Omega_f^+=\Omega_A^+$, the proof of Theorem~\ref{thm:BSDformulainrankoneIntro} follows. 
\end{proof}

\subsection{Proof of Theorem~\ref{thm:PRconjectureIntro} (Perrin-Riou's conjecture)}
\label{sec:PRsconj}
We continue to assume throughout \S\ref{sec:PRsconj}, as in \S\ref{subsec_PRbigLog_padicLBK}-\S\ref{sec:BSD}, that $f=f_A$ is an eigenform of weight $2$, which is associated to a non-CM elliptic curve $A_{/\QQ}$ that has good ordinary reduction at $p$ and that has analytic rank one. We also assume that $\overline{\rho}_A$ is absolutely irreducible. 

We shall follow the argument in the proof of Theorem 2.4(iv) of \cite{kbbIwasawa2017} very closely, where the analogous assertion has been verified in the case when the prime $p$ is a prime of good supersingular   reduction. Essentially, the argument in op.\ cit.\ works verbatim, on replacing all references to Kobayashi's work with references to  Corollary~\ref{cor:padicGZoverQIntro}, Perrin-Riou's $p$-adic Gross--Zagier formula at slope zero and Theorem~\ref{thm:padicheightnontrivialIntro}. We summarize it here for the convenience of the readers. 

Let us write
$$\mathfrak{Log}_{V}=\mathfrak{Log}_{V,\alpha}\cdot \omega_\alpha+\mathfrak{Log}_{V,\beta}\cdot\omega_\beta\,.$$
Let $\omega^*_{\mathscr{A}} \in \DD_{\textup{cris}}(V)/\textup{Fil}^0\DD_{\textup{cris}}(V)$ denote the unique element such that $[\omega_{\mathscr{A}},\omega^*_{\mathscr{A}}]=1$, where
$$[\,,\,]:\, \DD_{\textup{cris}}(V) \otimes \DD_{\textup{cris}}(V) \lra \QQ_p$$
is the natural pairing. We define  $\log_A(\res_p({\rm BK}_1))$ according to the identity 
$$\log_{V}(\res_p({\rm BK}_1))=\log_A(\res_p({\rm BK}_1))\cdot\omega^*_{\mathscr A}\,.$$
The dual basis of $\{\omega_\alpha,\omega_\beta\}$ with respect to the pairing $[\,,\,]$ is $\{\omega_\beta^*,\omega_\alpha^*\}$, where $\omega_\beta^*$ (respectively, $\omega_\alpha^*$) is the image of $\omega^*_{\mathscr A}$ under the inverse of the isomorphism 
$$s_{D_\beta}:\,D_\beta\stackrel{\sim}{\ra} \DD_{\textup{cris}}(V)/\textup{Fil}^0\DD_{\textup{cris}}(V)$$ (respectively, under the inverse of $s_{D_\alpha}$). Let $\lambda\in \{\alpha,\beta\}$ be such that the height pairing $h_{\lambda,\QQ}^{\rm Nek}$ is non-trivial and let $\lambda^*$ be given so that $\{\lambda,\lambda^*\}=\{\alpha,\beta\}$. Then,
\begin{align}
\notag(1-\frac{1}{\lambda})^2\cdot c(f)\cdot h_{\lambda,\QQ}^{\rm Nek}(P_f,P_f)&=L_{p,\lambda}^\prime(f,1)\\
\notag&=\mathds{1}\left(\mathfrak{Log}_{V,\lambda}(\partial_\lambda\mathbb{BK}_1)\right)\\
\notag&=\left[\exp^*(\partial_\lambda{{\rm BK}}_1),(1-p^{-1}\varphi^{-1})(1-\varphi)^{-1}\cdot\omega_{\lambda^*}^*\right]\\
\notag&=(1-\frac{\lambda^*}{p})(1-\frac{1}{\lambda^*})^{-1}\left[\exp^*(\partial_\lambda{{\rm BK}}_1),\omega^*_{\mathscr A}\right]\\
\notag&=(1-\frac{1}{\lambda})(1-\frac{1}{\lambda^*})^{-1}\frac{\left[\exp^*(\partial_\lambda{{\rm BK}}_1),\log_V(\res_p({\rm BK}_1))\right]}{\log_A(\res_p({\rm BK}_1))}\\
\notag&=-(1-\frac{1}{\lambda})(1-\frac{1}{\lambda^*})^{-1}\,\frac{h_{\lambda,\QQ}^{\rm Nek}\left({\rm BK}_1,{\rm BK}_1\right)}{\log_A(\res_p({\rm BK}_1))}.
\end{align}
Here:
\begin{itemize}
\item The first equality follows from Perrin-Riou's $p$-adic Gross--Zagier formula if $\lambda =\alpha$, or else it is Corollary~\ref{cor:padicGZoverQIntro}.
\item The second equality follows from the definition of $\mathfrak{Log}_{V,\lambda}$ and the fact that it maps Beilinson-Kato class (determined with the choice $\Omega_f^\pm=\Omega_A^\pm$, as in \S\ref{subsec_PRbigLog_padicLBK}) to $\mathscr{D}_{f^\lambda_{/\QQ}}$ (Theorem~\ref{thm:katohansen}); as well as the definition of the \emph{derived Beilinson-Kato class} 
$$\partial_\lambda\mathbb{BK}_1 \in H^1_\Iw(\widetilde{\DD}_\lambda)$$ 
which is given within the proof of Theorem~2.4 in \cite{kbbIwasawa2017}.
\item The element 
$$\partial_\lambda{{\rm BK}}_1\in H^1_{/\rm f}(\QQ_p,V):=H^1(\QQ_p,V)/H^1_{\rm f}(\QQ_p,V)$$ 
is the projection of the derived Beilinson-Kato class $\partial_\lambda\mathbb{BK}_1$ under the natural map 
$${\rm pr}_0\,:\, H^1_\Iw(\widetilde{\DD}_\lambda)\lra H^1_{/\rm f}(\QQ_p,V)$$ and the third equality follows from the explicit reciprocity laws of Perrin-Riou (as proved by Colmez) (c.f. the discussion in \cite[Section 2.1]{kbbleiintegralMC}). 
\item Fourth and fifth equalities follow from definitions (and using the fact that $\lambda\lambda^*=p$). 
\item The final equality follows from the Rubin-style formula proved in \cite[Theorem 4.13]{BenoisBuyukbodukExceptionalPR} and the comparison of various $p$-adic heights summarized in the diagram \eqref{eqn:comaprisonofheights}.
\end{itemize}
Recalling from \eqref{eqn_P_vs_Pf_final_section} the definition of the rational multiplier $d$ and that $d(f)=d^2c(f)$, we infer that
\begin{equation}
\label{eqn:BKcalculation}
\frac{h_{\lambda,\QQ}^{\rm Nek}\left({\rm BK}_1,{\rm BK}_1\right)}{\log_A(\res_p({\rm BK}_1))}=-(1-1/\alpha)(1-1/\beta)\cdot d(f)\cdot h_{\lambda,\QQ}^{\rm Nek}(P,P)\,.
\end{equation}
It follows from the fact that $h_{\lambda,\QQ}^{\rm Nek}(\ast,\ast)$ and $\left(\log_A\circ\,\res_p\,(\ast)\right)^2$ are both non-trivial quadratic forms on  the one dimensional $\QQ_p$-vector space $A(\QQ)\otimes\QQ_p$, combined with \eqref{eqn:BKcalculation} that 
$$\frac{h_{\lambda,\QQ}^{\rm Nek}(P,P)}{\log_A(\res_p(P))^2}=\frac{h_{\lambda,\QQ}^{\rm Nek}({\rm BK}_1,{\rm BK}_1)}{\log_A(\res_p({\rm BK}_1))^2}=-(1-1/\alpha)(1-1/\beta)\cdot d(f)\cdot \frac{h_{\lambda,\QQ}^{\rm Nek}(P,P)}{\log_{A}(\res_p({\rm BK}_1))}\,.$$
\qed
%%%%%%%%%%%%%%%%%%%%%%%%%%%%%%%%%%%%%%%%%%%%%%%%%%%%%%%%%%%%%%%%%%%%%%%%%%%%%%%%%%%%%%%%%%%%%%%%%%%%%%%%%%%%%%%%%%%%%%%%%%%%%%%%%%%%%%%%%%%%%%%%%%%%%%%%%%%%%%%%%%%%%%%%%%%%%%%%%%%%%%%%%%%%%%%%%%%%%%%%%%%%%%%%%%%%%%%%%%%%%%%%%%%%%%%%%%%%%%%%%%%%%%%%%%%%%%%%%%%%%%%%%%%%%%%%%%%%%%%%%%%%%%%%%%%%%%%%%%%%%%%%%%%%%%%%%%%%%%

\bibliographystyle{amsalpha}
\bibliography{references}

\end{document}